\documentclass[a4paper,12pt]{article}
\usepackage{amsthm}
\usepackage{amsmath,amsfonts,amssymb}
\usepackage{a4wide}
\usepackage[usenames,dvipsnames]{color}
\usepackage[all,cmtip]{xy}
\usepackage[verbose,colorlinks=true,linktocpage=true,linkcolor=blue,citecolor=blue]{hyperref}
\usepackage{footnote}
\usepackage{tikz}
\usetikzlibrary{arrows}
\usetikzlibrary{cd}
\usepackage{extarrows}%2023.01.25 arrows
\usepackage{cite}
\theoremstyle{definition}
\newtheorem{definition}{Definition}[section]

\theoremstyle{plain}
\newtheorem{theorem}[definition]{Theorem}
\newtheorem{proposition}[definition]{Proposition}
\newtheorem{lemma}[definition]{Lemma}
\newtheorem{corollary}[definition]{Corollary}

\theoremstyle{remark}
\newtheorem{remark}[definition]{Remark}
\usepackage{tabto}

\numberwithin{equation}{section}

\newcommand{\C}{\mathbb{C}}

\newcommand\keywords[1]{\textbf{keywords}:#1}

\allowdisplaybreaks[1]

\begin{document}
\date{}
\title{RTT presentation of coideal subalgebra of quantized enveloping algebra of type CI}
\author{Yingwen Zhang${}^1$  Hongda Lin${}^2$ and Honglian Zhang${}^{1}$\thanks{Corresponding Author. Email:\ hlzhangmath@shu.edu.cn}~~}
\maketitle

\begin{center}
\footnotesize
\begin{itemize}
\item[1] Department of Mathematics, Shanghai University, Shanghai 200444, China.
\item[2] Shenzhen International Center for Mathematics, Southern University of Science and Technology, China.
%\item[3] Newtouch Center for Mathematics at Shanghai University, Shanghai 200444, China.
\end{itemize}
\end{center}

\begin{abstract}
	The pair consisting of a quantum group and its corresponding coideal subalgebra, known as a quantum symmetric pair, was developed independently by M. Noumi and G. Letzter through different approaches. The purpose of this paper is threefold. First, for symmetric pairs $(\mathfrak{sp}_{2n},\mathfrak{gl}_n)$, we construct a coideal subalgebra $U_q^{tw}(\mathfrak{gl}_n)$ of the quantized enveloping algebra of type CI using the $R$-matrix presentation, based on the work of Noumi. Second, we derive a Poincar\'e-Birkhoff-Witt(PBW) basis for $U_q^{tw}(\mathfrak{gl}_n)$ by  the $\mathbb{A}$-form approach. As a consequence of the isomorphism btween $U_q^{tw}(\mathfrak{gl}_n)$ and the $\imath$quantum group $\mathcal{U}^{\imath}$, our method also yields the PBW basis for the $\imath$quantum group of type CI. Finally, as an application of the $R$-matrix presentation, we construct a Poisson algebra $\mathcal{P}_n$ associated with $U_q^{tw}(\mathfrak{gl}_n)$, and explicitly describe the action of the braid group $\mathcal{B}_n$ on the elements of $\mathcal{P}_n$.
\end{abstract}

\keywords  {\ Quantized enveloping algebra $\cdot $ Coideal subalgebra $\cdot$ Quantum symmetric pair $\cdot $ Poisson algebra}

\section{Introduction}
Quantum groups, introduced to solve integrable models and construct solutions to the Yang-Baxster equation, were independently established as Hopf algebras by V. G. Drinfeld \cite{V.G.} and M. Jimbo \cite{M.J.} around 1985. In many instances, quantum groups $U_q(\mathfrak{g})$ are viewed as deformations of universal enveloping algebras of Lie algebras $\mathfrak{g}$. 
An alternative definition of \( U_q(\mathfrak{g}) \) is given by the RTT presentation, introduced by L. D. Faddeev, N. Y. Reshetikhin, and L. A. Takhtajan \cite{FR}.
The theory of quantum groups has since found widespread application across numerous areas of mathematics and mathematical physics, including solvable lattice models, representation theory, and topology invariants, among others.

It is well-known that for a subalgebra $\mathfrak{t}$ of Lie algebra $\mathfrak{g}$, $U(\mathfrak{t})$ is a Hopf subalgebra of $U(\mathfrak{g})$. However, the $q$-analog $U_q(\mathfrak{t})$ is not necessarily isomorphic to a Hopf subalgebra of  $U_q(\mathfrak{g})$. To address this issue, in the 1990s, M. Noumi and T. Sugitani introduced the twisted quantized enveloping algebra $U_{q}^{tw}(\mathfrak{t})$  as coideal subalgebras of quantized universal enveloping algebras for all classical symmetric pairs $(\mathfrak{g}, \mathfrak{t})$ in \cite{N2}. Their construction is mainly achieved through the explicit solutions of the reflection equation, along with the RTT presentation of the quantized enveloping algebra.   In 1995,  M. Noumi \cite{N1} further studied  the twisted quantized enveloping algebras $U_q^{tw}(\mathfrak{o}_n)$ and $U_q^{tw}(\mathfrak{sp}_{2n})$ corresponding to symmetric pairs of type A. In 2003, A. Molev et al. \cite{M1} provided the defining relations of specific generators and the PBW basis for $U_q^{tw}(\mathfrak{o}_n)$ and $U_q^{tw}(\mathfrak{sp}_{2n})$.  The PBW basis for $U_q^{tw}(\mathfrak{sp}_{2n})$ is also provided in references \cite{M2,LK,LWZ24}.  Moreover, A. Molev et al. extended the twisted $q$-Yangian $\mathrm{Y}_q^{tw}(\mathfrak{o}_n)$ and $\mathrm{Y}_q^{tw}(\mathfrak{sp}_{2n})$ as coideal subalgebras of quantum affine algebra $U_q(\hat{\mathfrak{gl}}_n)$ in \cite{M1}. Additionally,  A. Molev and L. Gow classified the finite-dimensional irreducible representations of $U_q^{tw}(\mathfrak{sp}_{2n})$ and $\mathrm{Y}_q^{tw}(\mathfrak{sp}_{2n})$ of type $\mathrm{A\MakeUppercase{\romannumeral 2}}$, in terms of their highest weights in \cite{M2,M3}. 

However, the development of the BCD type remains insufficient, as the structure of the BCD type is more complex than that of the A type. In this paper, we define the twisted quantized enveloping algebra $U_q^{tw}(\mathfrak{gl}_n)$ of type CI, using the RTT-type relation and reflection equation. The crucial step in defining the coideal subalgebra is deriving the corresponding additional relation. Current research mainly focuses on the twisted Yangian, with significant contributions from \cite{GV,GRW,GRW2}. Due to the distinct properties of Yangian, the center discussed in \cite{GV} cannot be generalized to the context of the present work.

It is worth noting that in a series of papers by G. Letzter\cite{L1,L2,L3}, the unified theory of quantum symmetric pairs $(\mathcal{U},\mathcal{U}^\imath)=(\mathcal{U}_q(\mathfrak{g}), \mathcal{U}^\imath(\mathfrak{g},\mathfrak{t}))$ was established using the Drinfeld-Jimbo presentation of quantized enveloping algebra. The algebra $\mathcal{U}^{\imath}$ is the maximal right coideal subalgebra of $\mathcal{U}$, and it is now referred to as the $\imath$quantum group. In \cite{K14}, S. Kolb further generalized Letzter's theory to the Kac-Moody case. In recent years, the theory has received great attention, with substantial contributions such as \cite{W1,BKW18,CLW21,W2,FL15} etc.

In fact, there exists a close connection between the two constructions of M. Noumi and G. Letzter.  The coideal property ensures that the construction of the maximal one-sided coideal is essentially unique. In Section 6 of\cite{L1}, G. Letzter discussed that how $\mathcal{U}^{\imath}$ contains the coideal subalgebra obtained from the reflection equation, as constructed by M. Noumi. S. Kolb \cite{K08} demonstrated that the central elements in G. Letzter’s quantum group analogues of symmetric pairs lead to solutions of the reflection equation. In particular, matrix solutions provide one-dimensional representations
(or characters) of these algebras \cite{KS09}. More  recently, K. Lu established a direct and explicit isomorphism between the twisted $q$-Yangians in the $R$-matrix presentation and affine $\imath$quantum groups in the current presentation, associated with the symmetric pair of type AI in \cite{LK}. 

The symmetries of twisted quantized algebra and their associated Poisson algebra have been the subject of extensive study by numerous authors. As $q \to 1$, $U_q^{tw}(\mathfrak{o}_3)$ equipped with a Poisson bracket becomes a Poisson algebra,  as demonstrated in  \cite{NRZ} by J.E. Nelson, T. Regge and F. Zertuche.  In the context of stock matrices, $U_q^{tw}(\mathfrak{o}_n)$ leads to a Poisson algebra, with a detailed Poisson bracket provided in \cite{NR}.  The automorphisms of both the algebra $U_q^{tw}(\mathfrak{o}_6)$ and the Poisson bracket in the associated Poisson algebra are described with explicit group relations in \cite{NR,NR93}. An action of the braid group $\mathcal{B}_n$ on the Poisson algebra was introduced by B. Dubrovin \cite{BD} and A. I. Bondal\cite{AIB}. Moreover, in \cite{M08} A. Molev and E. Ragoucy presented a quantized action of  $\mathcal{B}_n$ on the twisted quantized enveloping algebra $U_q^{tw}(\mathfrak{o}_n)$, and constructed a Poisson algebra associated with $U_q^{tw}(\mathfrak{sp}_{2n})$. In addition, they derived explicit formulas for the Poisson bracket in terms of the reflection equations.   
 
In the current paper, we construct a coideal subalgebra $U_q^{tw}(\mathfrak{gl}_n)$ of the quantized enveloping algebra $U_q(\mathfrak{sp}_{2n})$ with reflection equation, corresponding to the symmetric pair:
$$\mathrm{C\MakeUppercase{\romannumeral 1}}:\ (\mathfrak{sp}_{2n}, \mathfrak{gl}_n).$$ Additionally, we provide the defining relations for specific generators. Furthermore,  we establish the Poincaré-Birkhoff-Witt (PBW) theorem for $U_q^{tw}(\mathfrak{gl}_n)$ using the method stated in \cite{M1}, as presented in Theorem \ref{UPbw}. However, the proof of the PBW theorem is more intricate due to the presence of central relations, as discussed in Section 4.
Moreover, we describe an isomorphism between the twisted quantized enveloping algebra $U^{tw}_q({\mathfrak{gl}_n})$ and the corresponding $\imath$quantum group $\mathcal{U}^{\imath}$ with generators $B_i =f_i-k_i^{-1}e_i$, for all $i\leq n$, as defined in \cite{KP11}, and present it in Theorem \ref{sisoui}. As a result, we also derive the PBW basis for the $\imath$quantum group of type C for the isomorphism between $U_q^{tw}(\mathfrak{gl}_{2n})$ and $\mathcal{U}^{\imath}$. 
In fact, Y. Xu in \cite{XY14} gave the PBW theorem for the $\imath$quantum group with generators $B_i$ of each type. The PBW basis of the $\imath$quantum group presented in the current paper through the isomorphism, differs slightly from the one in \cite{XY14}.

Furthermore, we construct a Poisson algebra $\mathcal{P}_n$ associated with the twisted quantized enveloping algebra $U_q^{tw}(\mathfrak{gl}_{n})$ by taking the limit $q \to 1$. Using the reflection relations, we derive an explicit formula for the Poisson bracket on the corresponding matrix space. 
Braid group action is a fundamental construction for usual quantum groups. Due to the isomorphism between $U^{tw}_q({\mathfrak{gl}_n})$ and the $\imath$quantum group $\mathcal{U}^{\imath}$ of type CI, we present the action of the braid group $\mathcal{B}_n$ on $U_q^{tw}(\mathfrak{gl}_{n})$ referring to \cite{KP11}. In addition, we explicitly describe the action of $\mathcal{B}_n$ on the elements of Poisson algebra $\mathcal{P}_n$, which preserves the Poisson bracket.

The paper is organized as follows. In Section 2, we recall the Drinfeld-Jimbo and RTT presentation of the quantized universal enveloping algebra $U_q(\mathfrak{sp}_{2n})$, along with the isomorphism between these two presentations. In Section 3, we first review the definition and properties of the symmetric pair $(\mathfrak{sp}_{2n},\mathfrak{gl}_n$) and define the twisted quantized universal enveloping algebra $U^{tw}_q({\mathfrak{gl}_n})$ as coideal subalgebra of $U_q({\mathfrak{sp}_{2n}})$. Moreover, we provide an equivalent presentation of $U^{tw}_q({\mathfrak{gl}_n})$. Section 4 is devoted to proving the PBW theorem of $U^{tw}_q({\mathfrak{gl}_n})$. In Section 5, we establish an isomorphism between the twisted quantized universal enveloping algebra $U^{tw}_q({\mathfrak{gl}_n})$ and the corresponding $\imath$quantum group $\mathcal{U}^{\imath}$,  and provide a PBW basis for the $\imath$quantum group as an application. Finally, in Section 6, we construct a new Poisson algebra associated with $U^{tw}_q({\mathfrak{gl}_n})$, and explicitly present the action of the braid group $\mathcal{B}_n$ on the elements of $\mathcal{P}_n$.

Throughout this paper, all spaces and algebras are over the complex set $\C$. We denote by $\delta_*$ which equals 1 if the condition $*$ is satisfied and 0 otherwise.

\section{The quantized enveloping algebra for type $\mathfrak{sp}_{2n}$}
 We begin by reviewing the structure of the quantum algebra $\mathcal{U}_q(\mathfrak{g})$ introduced by V. G. Drinfeld \cite{V.G.} and M. Jimbo \cite{M.J.}.  In this section, we focus on
the symplectic Lie algebra $\mathfrak{sp}_{2n}$, a classical simple Lie algebra with Cartan matrix $A=(c_{i\,j})$. Fix an ordering set $\pi=
\{\alpha_1,\alpha_2,\cdots,\alpha_n\}$ of the simple roots, where 
$$\alpha_i=\lambda_i-\lambda_{i+1},\ \ \text{for} \ i=1,\cdots,n-1,\quad \text{and} \ \alpha_n=2\lambda_n,$$
$\lambda_1,\cdots\lambda_n$ is an orthonormal basis of a Euclidean space with bilinear product of $(\cdot, \cdot )$, and 
\begin{equation*}
	c_{i\,j}=\frac{2(\alpha_i,\alpha_j)}{(\alpha_i,\alpha_i)}.
\end{equation*}

Let $q$ be a formal variable which is nonzero and not a root of unity. Set $q_i=q^{d_i}$ for $i=1,\cdots,n$, where $d_i=(\alpha_i,\alpha_i)/2$. That is $q_i=q$ for $i<n$ and $q_n=q^2$.
The Drinfeld-Jimbo presentation $\mathcal{U}_q({\mathfrak{sp}_{2n}})$ of the quantized universal enveloping algebra for type ${\mathfrak{sp}_{2n}}$ is an associative algebra over $\C(q)$, generated by $e_i,f_i,k_i,k_i^{-1}$, $1\leq i \leq n$, subject to the defining relations
\begin{equation*}
	\begin{aligned}
		k_ik_j&=k_jk_i,\quad k_ik_i^{-1}=k_i^{-1}k_i=1,\\
		k_ie_jk_i^{-1}&=q_i^{c_{i\,j}}e_j,\quad k_if_jk_i^{-1}=q_i^{-c_{i\,j}}f_j,\\
		e_if_j&-f_je_i=\delta_{i\,j}\frac{k_i-k_i^{-1}}{q_i-q_i^{-1}},\\
		\sum_{k=0}^{1-c_{i\,j}}(-1)^k& \begin{bmatrix} 1-c_{i\,j} \\ k \end{bmatrix}_{q_i}e_i^{1-c_{i,j}-k}e_je_i^k=0,\quad i\neq j,\\
		\sum_{k=0}^{1-c_{i,j}}(-1)^k& \begin{bmatrix} 1-c_{i\,j} \\ k \end{bmatrix}_{q_i}f_i^{1-c_{i\,j}-k}f_jf_i^k=0,\quad i\neq j,\\
	\end{aligned}	
\end{equation*}
where 
$$[k]_{q_i}=\frac{q_i^n-q_i^{-n}}{q_i-q_i^{-1}},\qquad [k]_{q_i}!=\prod_{s=1}^k[s]_{q_i},$$
$$\begin{bmatrix} k \\ r \end{bmatrix}_{q_i}=\frac{[k]_{q_i}!}{[r]_{q_i}![k-r]_{q_i}!}.$$
The algebra $\mathcal{U}_q({\mathfrak{sp}_{2n}})$ is equipped with a Hopf algebra structure referred in \cite{NL90}, where the comultiplication $\Delta$, counit $\varepsilon$ and antipode $\mathfrak{S}$ are defined as follows:
\begin{equation*}
	\begin{aligned}
		&\Delta(k_i)=k_i\otimes k_i,\ \Delta(k_i^{-1})=k_i^{-1}\otimes k_i^{-1},\ \varepsilon(k_i)=1,\ \mathfrak{S}(k_i)=k_i^{-1},\\
		&\Delta(e_i)=e_i\otimes k_i+1\otimes e_i,\quad \varepsilon(e_i)=0,\quad \mathfrak{S}(e_i)=-e_ik_i^{-1},\\
		&\Delta(f_i)=f_i\otimes1+k_i^{-1}\otimes f_i,  \quad \varepsilon(f_i)=0, \quad \mathfrak{S}(f_i)=-k_if_i.\\
	\end{aligned}	
\end{equation*}

Recall the $R$-matrix is an element of the endomorphism algebra End$\C^{2n}\otimes$ End$\C^{2n}$, given by
\begin{equation}\label{R}
	R=\sum_{i,j}q^{\delta_{i\,j}-\delta_{i\,j'}}e_{i\,i}\otimes e_{j\,j}+(q-q^{-1})\sum_{i<j}e_{i\,j}\otimes e_{j\,i}
	-(q-q^{-1})\sum_{i<j}q^{\overline{j}-\overline{i}}\epsilon_i\epsilon_je_{i\,j}\otimes e_{i'\,j'},
\end{equation}
where $e_{i\,j}\in$End$\C^{2n}$ denote the elementary matrices, $i'=2n+1-i$ and
$$\epsilon_i=\left\{
\begin{aligned}
	1 &    & i\leq n\\
	-1&    & i > n
\end{aligned} \right. ,$$
$$(\overline{1},\cdots,\overline{n},\overline{n+1},\cdots,\overline{2n})=(n,\cdots,1,-1,\cdots,-n).$$
Indeed, the $R$-matrix is a solution to the following Yang-Baxter equation
\begin{equation*}
	R_{12}R_{13}R_{23}=R_{23}R_{13}R_{12}
\end{equation*}
on End$\C^{2n}\otimes$End$\C^{2n}\otimes$End$\C^{2n}$, and the subindices indicate the copies of End$\C^{2n}$, e.g. $R_{12}=R\otimes 1$, etc.
\begin{definition} (\cite{NL90})
	The extented quantized enveloping algebra $U_q^{ext}(\mathfrak{sp}_{2n})$ is an associative algebra over $\C(q)$, generated by $t_{i\,j}$, $\overline{t}_{i\,j}$ for $1\leq i,j\leq 2n$. The defining relations of generators are given by 
	\begin{equation}\label{Rtt}
		\begin{aligned}
			&t_{i\,j}=\overline{t}_{j\,i}=0, \ \quad\quad 1\leq i<j\leq 2n,\\
			&t_{i\,i}\overline{t}_{i\,i}=\overline{t}_{i\,i}t_{i\,i}=1,\ \ \ 1\leq i\leq 2n,\\
			&RT_1T_2=T_2T_1R,\qquad R\overline{T}_1\overline{T}_2=\overline{T}_2\overline{T}_1R,\qquad R\overline{T}_1T_2=T_2\overline{T}_1R,\\
		\end{aligned}	
	\end{equation}
	where 
	\begin{equation*}
		T=\sum_{i\geq j}e_{i\,j}\otimes t_{i\,j}, \quad \overline{T}=\sum_{i\leq j}e_{i\,j}\otimes \overline{t}_{i\,j}.
	\end{equation*}
\end{definition}
The $2n\times 2n$ matrices $T$ and $\overline{T}$ are regarded as elements of the algebra End $\C^{2n} \otimes U_q^{ext}(\mathfrak{sp}_{2n})$ corresponding to the $L$-operators $L^-$ and $L^+$ respectively, as discussed in the \cite{NL90}. They act as root vectors of $\mathfrak{sp}_{2n}$. We denote by $T_a$ and $\overline{T}_a$ the a-th copies of End $\C^{2n}$ in the tensor product: $T_a=1^{\otimes(a-1)}\otimes e_{i\,j}\otimes 1^{\otimes(2-a)} \otimes  t_{i\,j}$, for $a=1,2$, with $\overline{T}_a$ defined similarly. 
Therefore, the specific relation from the RTT relation (\ref{Rtt}) between $t_{i\,j}$ can be rewritten as 
\begin{equation}\label{rtt1}
	\begin{aligned}
		q^{\delta_{i\,j}-\delta_{i\,j'}}t_{i\,a}t_{j\,b}-q^{\delta_{a\,b}-\delta_{a\,b'}}t_{j\,b}t_{i\,a}
		=&(q-q^{-1})(\delta_{b<a}t_{j\,a}t_{i\,b}-\delta_{i<j}t_{j\,a}t_{i\,b})\\
		&+\delta_{i\,j'}(q-q^{-1})\sum_{k=i+1}q^{\overline{k}-\overline{i}}\epsilon_i\epsilon_kt_{k\,a}t_{k'\,b}\\
		&-\delta_{a\,b'}(q-q^{-1})\sum_{k=1}^{a-1}q^{\overline{a}-\overline{k}}\epsilon_k\epsilon_at_{j\,k'}t_{i\,k}.\\
	\end{aligned}	
\end{equation}
The  definition relation for $\overline{t}_{i\,j}$ is obtained by replacing $t_{i\,j}$ with $\overline{t}_{i\,j}$ in equation (\ref{rtt1}). Moreover, we have the relation between $\overline{t}_{i\,j}$ and $t_{i\,j}$, given by
\begin{equation}\label{rtt2}
	\begin{aligned}
		q^{\delta_{i\,j}-\delta_{i\,j'}}\overline{t}_{i\,a}t_{j\,b}-q^{\delta_{a\,b}-\delta_{a\,b'}}t_{j\,b}\overline{t}_{i\,a}
		=&(q-q^{-1})(\delta_{b<a}t_{j\,a}\overline{t}_{i\,b}-\delta_{i<j}\overline{t}_{j\,a}t_{i\,b})\\
		&+\delta_{i\,j'}(q-q^{-1})\sum_{k=i+1}q^{\overline{k}-\overline{i}}\epsilon_i\epsilon_k\overline{t}_{k\,a}t_{k'\,b}\\
		&-\delta_{a\,b'}(q-q^{-1})\sum_{k=1}^{a-1}q^{\overline{a}-\overline{k}}\epsilon_k\epsilon_at_{j\,k'}\overline{t}_{i\,k}.\\
	\end{aligned}	
\end{equation}
In this paper, we use $U_q(\mathfrak{sp}_{2n})$ to denote the RTT presentation of the quantized universal enveloping algebra of type ${\mathfrak{sp}_{2n}}$, as defined \cite{NL90}.

\begin{definition} (\cite{NL90}) \label{t}
	The RTT presentation $U_q(\mathfrak{sp}_{2n})$ of the quantized universal enveloping algebra of type ${\mathfrak{sp}_{2n}}$  is defined as the quotient algebra of $U_q^{ext}(\mathfrak{sp}_{2n})$ by the relations 
	\begin{equation}\label{cent}
		\begin{aligned}
			&TDT^tD^{-1}=\overline{T}D\overline{T}^tD^{-1}=I,\\
			&DT^tD^{-1}T=D\overline{T}^tD^{-1}\overline{T}=I,\\
		\end{aligned}	
	\end{equation}
	where $t$ denotes the matrix transposition with $e_{i\,j}^t=\epsilon_{i}\epsilon_je_{j'\,i'}$ and $D$ is the diagonal matrix
	\begin{equation*}
	D=\text{diag}(q^n,\cdots,q,q^{-1},\cdots,q^{-n}).
	\end{equation*}
	$I$ is the identity matrix.
\end{definition}
\begin{remark}
Compared with the quantized enveloping algebras of type $\operatorname{A}$, those of type $\operatorname{BCD}$ involve more intricate central relations, which presented significant challenges in proving the PBW theorem of their coideal subalgebra.
\end{remark}
The algebra $U_q({\mathfrak{sp}_{2n}})$ is equipped with a Hopf algebra structure, where the comultiplication $\Delta'$, counit $\varepsilon'$ are defined as follows,
\begin{equation*}
	\begin{aligned}
	\Delta'(t_{i\,j})&=\sum_{k=1}^{2n}t_{i\,k}\otimes t_{k\,j},\quad \Delta'(\overline{t}_{i\,j})=\sum_{k=1}^{2n}\overline{t}_{i\,k}\otimes \overline{t}_{k\,j},\quad \varepsilon'(t_{i\,j})=\varepsilon'(\overline{t}_{i\,j})=\delta_{i\,j},
	\end{aligned}	
\end{equation*}

\begin{proposition} ([\citenum{NL90},Theorem 12.])\label{iso}
	There exists a Hopf algebraic isomorphism ${\phi}$ from  $U_q(\mathfrak{sp}_{2n})$ to  $\mathcal{U}_q(\mathfrak{sp}_{2n})$  such that 
	$$\phi: \quad t_{i\,i}\mapsto t_i,\qquad \overline{t}_{i\, i+1}\mapsto -(q_i-q_i^{-1})t_i^{-1}e_i,\qquad t_{i+1\, i}\mapsto (q_i-q_i^{-1})f_it_i,$$
	where $k_i=t_it_{i+1}^{-1}$, for $1\leq i\leq n$.
\end{proposition}

\section{Coideal subalgebra of quantized enveloping algebra for type $U_q(\mathfrak{sp}_{2n})$}
In this section, we focus on the symmetric pair of type CI: $(\mathfrak{sp}_{2n},\mathfrak{gl}_n$). 
It is well known that $\mathfrak{gl}_n$ is a fixed subalgebra of $\mathfrak{sp}_{2n}$, although no explicit construction is known, which will be provided below.
Based on the construction of the coideal algebra for type A,  as given in \cite{N1}, and the constant solution to the reflection equation to the symmetric pair $(\mathfrak{sp}_{2n},\mathfrak{gl}_n$), we provide a construction of the coideal subalgebra $U^{tw}_q({\mathfrak{gl}_n})$ for $U_q(\mathfrak{sp}_{2n})$, as well as its equivalent presentation to provide an isomorphism to the $\imath$quantum group $\mathcal{U}^{\imath}$ for type C.

\subsection{Symmetric pair $(\mathfrak{sp}_{2n},\mathfrak{gl}_n$)}
The Lie algebra $\mathfrak{gl}_n$ is the general linear Lie algebra 
 associated with the standard basis $e_{i\,j}$, where $e_{i\,j}$ denotes the usual elementary matrix.
In fact, the symplectic Lie algebra $\mathfrak{sp}_{2n}$ is generated by the elements $F_{i\,j}=e_{i\,j}-\epsilon_{i}\epsilon_{j}e_{j'\,i'}$, for $1\leq i,j \leq 2n$. 
Here, $\epsilon_{i}$ is denoted as in Section 2.
For our purpose, we define $J=\sum_k\epsilon_{k}e_{k\,k}$, which corresponds to a special case of the definition provided in \cite{N2}.  
Observe that $J$ is an $2n\times 2n$ invertible matrix such that $J = J^{-1}$. 
Let $\theta: \mathfrak{sp}_{2n}\to \mathfrak{sp}_{2n}$ be an involution such that
\[
\theta : X\mapsto JX^uJ^{-1},
\]
where $u$ denotes the usual transposition such that $e_{i\,j}^u = e_{j\,i}$. 

Consider the decomposition $\mathfrak{sp}_{2n}=\mathfrak{k}\oplus \mathfrak{p}$ into the $-1$ and the $+1$ eigenspace of $\theta$. 
Let $\mathfrak{sp}_{2n}^{\theta}$ be the fixed subalgebra of $\mathfrak{sp}_{2n}$ under a certain involution $\theta$, i.e., $\mathfrak{k}=\mathfrak{sp}_{2n}^{\theta}=\{x\in \mathfrak{sp}_{2n} | \theta(x)=-x \}$.
As a consequence, one has 
$$\mathfrak{sp}_{2n}^{\theta} = \text{span}_{\mathbb{C}}\{G_{i\,j}=\epsilon_{i}F_{i\,j}-\epsilon_{j}F_{j\,i}|1\leq i,j\leq 2n\}.$$
The elements $G_{i\,j}$ satisfy the bracket 
\begin{equation*}
    \begin{aligned}
        [G_{i\,j},G_{k\,l}]=&\epsilon_{j}(\delta_{j\,k}G_{i\,l}+\delta_{j'\,k}G_{i'\,l})-\epsilon_{i}(\delta_{i\,k}G_{j\,l}+\delta_{i'\,k}G_{j'\,l})\\
        &+\epsilon_{i}(\delta_{i\,l}G_{j\,k}+\delta_{i'\,l}G_{j'\,k})-\epsilon_{j}(\delta_{j\,l}G_{i\,k}+\delta_{j'\,l}G_{i'\,k}),\\
    \end{aligned}
\end{equation*}
where the bracket $[x,y]=xy - yx$. 
For $i, j = 1,\cdots,n$, it is easy to get,
$$G_{i'\,j'}=-G_{i\,j},\quad G_{j'\,i'}=G_{i\,j},\quad G_{j\,i}=-G_{i\,j}.$$
\begin{proposition} \label{propgl} 
 There exists an isomorphism of Lie algebras from the general linear algebra $\mathfrak{gl}_n$ to the algebra  $\mathfrak{sp}_{2n}^{\theta}$. 
 Let $\psi:\mathfrak{gl}_n\to \mathfrak{sp}_{2n}^{\theta}$ be the isomorphism such that for $1\leq i<j\leq n$,  it follows that,
    $$e_{i\,i}\mapsto\frac{G_{i'\,i}}{2},\quad e_{j\,j}\mapsto\frac{G_{j'\,j}}{2},\quad e_{i\,j}\mapsto(-1)^{i-j+1}\frac{G_{j\,i}-G_{j'\,i}}{2},\quad e_{j\,i}\mapsto(-1)^{i-j+1}\frac{G_{j\,i}+G_{j'\,i}}{2}.$$
\end{proposition}
\begin{proof}
A straightforward calculation using the basis shows that $\psi([x\,y])=[\psi(x),\psi(y)]$ for all $x,y \in \mathfrak{gl}_n$. 
Moreover, $\psi$ is both surjective and injective by definition.
\end{proof}
In this paper, we introduce a certain quantization of the symmetric pairs $(\mathfrak{sp}_{2n}^{\theta}, \mathfrak{sp}_{2n})$ associated with the involution $\theta$.

\subsection{\ Coideal subalgebra $U_q^{tw}(\mathfrak{gl}_n)$ and its PBW basis}
In this subsection, we introduce the coideal subalgebra $U_q^{tw}(\mathfrak{gl}_n)$ of the quantized universal enveloping algebra  $U_q(\mathfrak{sp}_{2n})$, which specializes to the general linear Lie algebra $\mathfrak{gl}_n$ as $q\to 1$.
Actually, the matrix $J$ satisfies the following reflection equation given in \cite{N2}.
\begin{equation}\label{ref}
    RJ_1R^uJ_2=J_2R^uJ_1R
\end{equation}
in End$\mathbb{C}^{2n}\otimes$ End$\mathbb{C}^{2n}$. Here, $R^u=R^{u_1}$ denotes the matrix obtained from $R$ by the action of the transposition $u$ on the first tensor factor,
\begin{equation*}
    R^u=\sum_{i,j}q^{\delta_{i\,j}-\delta_{i'\,j'}}e_{i\,i}\otimes e_{j\,j}+(q - q^{-1})\sum_{i<j}e_{j\,i}\otimes e_{j\,i}
    -(q - q^{-1})\sum_{i<j}q^{\overline{j}-\overline{i}}\epsilon_i\epsilon_je_{j\,i}\otimes e_{i'\,j'}.
\end{equation*}
Using the matrix $J$, we define the $2n\times2n$ matrix $S=TJ\overline{T}^u=(s_{i\,j})$, where  $s_{i\,j}=\sum_{k=j}^{i}\epsilon_kt_{i\,k}\overline{t}_{j\,k}$. 
\begin{definition}\label{defextu}
    The twisted quantized universal enveloping algebra $U_q^{tw}(\mathfrak{gl}_{n})$ is the subalgebra of $U_q(\mathfrak{sp}_{2n})$ generated by the entries of matrix $S$.
\end{definition} 
Actually, the twisted quantized enveloping algebra $U_q^{tw}(\mathfrak{gl}_n)$ is  a coideal subalgebra of $U_q(\mathfrak{sp}_{2n})$ given below.
\begin{proposition}
The twisted quantized enveloping algebra $U_q^{tw}(\mathfrak{gl}_n)$ is a left coideal subalgebra of $U_q(\mathfrak{sp}_{2n})$. To be more precise, the image of the generators $s_{i\,j}$ under the coproduct $\Delta'$ is given by 
    \begin{equation*}
    \Delta'(s_{i\,j})=\sum_{k,l}^{2n}t_{i\,k}\overline{t}_{j\,l}\otimes s_{k\,l}.
    \end{equation*}
\end{proposition}
\begin{proof}
    \begin{equation*}
        \begin{aligned}
            \Delta'(s_{i\,j})=&\Delta'(\sum_{s=1}^{2n}\epsilon_st_{i\,s}\overline{t}_{j\,s})=\sum_{s=1}^{2n}\epsilon_s\Delta'(t_{i\,s})\Delta'(\overline{t}_{j\,s})\\
            &=\sum_{s=1}^{2n}\epsilon_s(\sum_{k=1}^{2n}t_{i\,k}\otimes t_{k\,s})(\sum_{l=1}^{2n}\overline{t}_{j\,l}\otimes \overline{t}_{l\,s})
            =\sum_{k,l}^{2n}t_{i\,k}\overline{t}_{j\,l}\otimes s_{k\,l}.
        \end{aligned}
    \end{equation*}
    So the assertion is clear.
\end{proof}
Furthermore,it follows from direct calculation that the matrix transpositions $u$ and $t$, given in the definition (\ref{t}), are commutative. 
Let $ut$ denote the composite transposition of $t$ and $u$, so $e_{i\,j}^{ut}=\epsilon_{i}\epsilon_{j}e_{i'\,j'}$.
\begin{lemma}\label{overss}
Define the matrix $\overline{S}$ as follows,
\begin{align*}
\overline{S}=\overline{T}^{ut}J^tT^t=\sum_{i\,j}e_{i\,j}\otimes\overline{s}_{i\,j},    
\end{align*}
where 
$\overline{s}_{i\,j}=\epsilon_{i}\epsilon_{j}\sum_{k=j}^i\epsilon_{k'}\overline{t}_{i'\,k'}t_{j'\,k'}$. 
$\overline{s}_{i\,j}$ can be spanned by $s_{i\,j}$ as follows,
    \begin{align*}
        &\overline{s}_{i\,i} =s_{i'\,i'}=\epsilon_{i'},\qquad \quad \overline{s}_{i\,j} =0,\quad i<j,\\
        &\overline{s}_{i\,j} =qs_{i\,j}^t=q\epsilon_{i}\epsilon_{j}s_{j'\,i'}, \quad j<i,\  i\neq j', \\\
        &\overline{s}_{j'\,j} =-q^2s_{j'\,j}+(q^2-1)\sum_{m=j+1}^nq^{\overline{m}-\overline{j}}\overline{s}_{m'\,m},\quad j<n.
    \end{align*} 
It is easy to see that $\ \overline{s}_{j'\,j}$ is spanned by $s_{m'\,m},\ j\leq m\leq n$ by induction.
\end{lemma}
\begin{proof} 
According to the definition of $\overline{S}$, we have 
$\overline{s}_{i\,i}=\epsilon_{i'}$ and $\overline{s}_{i\,j} =0$ for $i<j$. 
Moreover, we can easily derive the following equation 
    \begin{equation*}\label{ssij}
        a\overline{s}_{i\,j}^t-bs_{i\,j}=\sum_{k=j}^{i}\epsilon_{k}(a\overline{t}_{j\,k}t_{i\,k}-bt_{i\,k}\overline{t}_{j\,k}), \quad a,b\in \C(q).
    \end{equation*}
    However, due to (\ref{rtt2}), we conclude that $$\overline{t}_{j\,k}t_{i\,k}-qt_{i\,k}\overline{t}_{j\,k}=-(q-q^{-1})(\overline{t}_{i\,k}t_{j\,k})=0, \quad j\leq k \leq i, \ i\neq j\neq j',$$ which implies $\overline{s}_{i\,j}^t-qs_{i\,j}=0$, $j<i,i\neq j'$. From the equation
    \begin{equation*}
        \begin{aligned}
            q^{-1}\overline{s}_{j'\,j}^t-qs_{j'\,j}=&\sum_{k=j}^{j'}\epsilon_{k}(q^{-1}\overline{t}_{j\,k}t_{j'\,k}-qt_{j'\,k}\overline{t}_{j\,k})
        =(q-q^{-1})\sum_{k=j}^{j'}\epsilon_{k}(\sum_{l=j+1}^{\text{min}\{k,k'\}}q^{{\overline{l}}-\overline{j}}\epsilon_{j}\epsilon_l\overline{t}_{l\,k}t_{l'\,k})\\
        =&(q-q^{-1})(q^{\overline{j+1}-\overline{j}}\sum_{k=j+1}^{(j+1)'}\epsilon_{k}\overline{t}_{j+1\,k}t_{(j+1)'\,k}+\cdots+q^{\overline{n}-\overline{j}}\sum_{k=n}^{n'}\epsilon_{k}\overline{t}_{n\,k}t_{n'\,k})\\
        =&(q-q^{-1})\sum_{m=j+1}^{n}q^{\overline{m}-\overline{j}}\overline{s}_{m'\,m}^t,
        \end{aligned}
    \end{equation*}
    it follows that
    \begin{equation*}
        \begin{aligned}\label{ssii'}
        \overline{s}_{j'\,j}=-q^2s_{j'\,j}+(q^2-1)\sum_{m=j+1}^{n}q^{\overline{m}-\overline{j}}\overline{s}_{m'\,m},\quad j\leq n.
        \end{aligned}
    \end{equation*}
We can directly derive $\overline{s}_{n'\,n}=-q^2s_{n'\,n}$.
Therefore, it can be deduced that $$\overline{s}_{j'\,j}=\sum_{m=j}^{n} (-1)^{k_m}q^{k_m}s_{m'\,m}$$ with powers $k_m$  by induction on $j$ for $j\leq n$.
\end{proof}

\begin{proposition}
  The matrix $S$ satisfies the following relations.
\begin{align}
    &s_{i\,j}=0,\qquad 1\leq i<j\leq 2n,\label{sij}\\
    &s_{i\,i}=\epsilon_i,\qquad1\leq i\leq 2n,\label{sii}\\
    &RS_1R^uS_2=S_2R^uS_1R,\label{srefection}\\
    &\overline{S}D^{-1}SD=-I.\label{Scentre}
\end{align}  
\end{proposition}
\begin{proof}

It is easy to see that $s_{i\,i}=\epsilon_i, 1\leq i\leq 2n$, and $s_{i\,j}=0, 1\leq i<j\leq 2n.$ 
From the reflection equation (\ref{ref}) and the following consequences
\begin{equation*}
    RJR^u=R^uJR,\quad \overline{T}_1^uR^uT_2=T_2R^u\overline{T}^u_1,\quad R\overline{T}_1^u\overline{T}_2^u=\overline{T}_2^u\overline{T}_1^uR,
\end{equation*}
we assert 
\begin{equation}\label{RTJ}
    RT_1J_1\overline{T}_1^uR^uT_2J_2\overline{T}_2^u=\overline{T}_2^uJ_2T_2R^u\overline{T}_1^uJ_1RT_1.
\end{equation}
The matrix $S$ satisfies the relation $RS_1R^uS_2=S_2R^uS_1R$ from (\ref{RTJ}).
   Due to (\ref{cent}), and following straightforward calculations, one obtains
    \begin{equation}\label{TDTD}
        \begin{aligned}
        T^uDT^{ut}D^{-1}=\overline{T}^uD\overline{T}^{ut}D^{-1}=I,\\
        DT^{ut}D^{-1}T^u=D\overline{T}^{ut}D^{-1}\overline{T}^u=I.\\
        \end{aligned}    
    \end{equation}
    It follows from (\ref{TDTD}) that
    \begin{equation*}
        \overline{T}^{ut}J^tT^tD^{-1}TJ\overline{T}^uD=-I,
    \end{equation*}
    which implies $\overline{S}D^{-1}SD=-I$.
\end{proof}
\begin{remark}
Indeed, the relations (\ref{sij})--(\ref{Scentre}) are precisely  the defining relation of the twisted quantized universal enveloping algebra $U_q^{tw}(\mathfrak{gl}_n)$. For further details please refer to  Theorem \ref{UPbw}. 
\end{remark}
 \begin{lemma}
We have the following specific relations
\begin{align}
    &\overline{s}_{i\,i}s_{i\,i}=\epsilon_{i'}\epsilon_{i}=-1, \quad & i= j,\label{siii'i'}\\
    &\sum_{k=j+1,k\neq i'}^{i-1}q^{\overline{j}-\overline{k}+1}\epsilon_{i}\epsilon_{k}s_{k'\,i'}s_{k\,j}+q\epsilon_is_{j'\,i'}+\epsilon_{i'}q^{\overline{j}-\overline{i}}s_{i\,j}+\delta_{j\leq i'\leq i}q^{\overline{j}+\overline{i}}\overline{s}_{i\,i'} s_{i'\,j}=0,\quad & i > j \label{scenspecific2}.
\end{align}   
 \end{lemma}
\begin{proof}
It follows from (\ref{Scentre}) that
        \begin{equation} \label{nScen}
        e_{i\,j}\otimes \sum_{k=j}^{i}q^{\overline{j}-\overline{k}}\overline{s}_{i\,k}s_{k\,j}=-I.
        \end{equation}
        Moreover, by Lemma \ref{overss}, the specific relations (\ref{siii'i'}) and (\ref{scenspecific2}) can be given by direct calculation. This lemma will be used in Section 4.
\end{proof}

%We provide the definition with specific relations for $U_q^{tw}(\mathfrak{gl}_n)$ as follows.

%\begin{definition}\label{def:S}
%   The twisted quantized enveloping algebra $U_q^{tw}(\mathfrak{gl}_n)$ is the quotient of $\widetilde{U}$ with the ideal generated by $\overline{S}D^{-1}SD$, i.e.,
%  \begin{equation*}
%     U_q^{tw}(\mathfrak{gl}_n)=\widetilde{U}/<\overline{S}D^{-1}SD>.
%\end{equation*}
%    The algebra $U_q^{tw}(\mathfrak{gl}_n)$ is a subalgebra of $U_q(\mathfrak{sp}_{2n})$ generated by the matrix element $S=TJ\overline{T}^u$.
%\end{definition}

Rewriting (\ref{srefection}) in terms of generators, we obtain the specific relations involving  $s_{i\,j}$,
\begin{equation}\label{sre}
    \begin{aligned}
        q^{\delta_{i\,j}-\delta_{i\,j'}+\delta_{a\,j}-\delta_{a\,j'}}s_{i\,a}s_{j\,b}
        =&q^{\delta_{i\,b}-\delta_{i\,b'}+\delta_{a\,b}-\delta_{a\,b'}}s_{j\,b}s_{i\,a}\\
        &+(q-q^{-1})q^{\delta_{i\,a}-\delta_{i\,a'}}(\delta_{b<a}-\delta_{i<j})s_{j\,a}s_{i\,b}\\
        &+(q-q^{-1})(q^{\delta_{a\,b}-\delta_{a\,b'}}\delta_{b<i}s_{j\,i}s_{b\,a}-q^{\delta_{i\,j}-\delta_{i\,j'}}\delta_{a<j}s_{i\,j}s_{a\,b})\\
        &+(q-q^{-1})^2(\delta_{b<a<i}-\delta_{a<i<j})s_{j\,i}s_{a\,b}\\
        &+(q-q^{-1})^2\delta_{a\,i'}\sum_{k=i'+1}q^{\overline{k}-\overline{i'}}\epsilon_i\epsilon_{k'}(\delta_{i<j}-\delta_{b<i'})s_{j\,k}s_{k'\,b}\\
        &-(q-q^{-1})\delta_{b\,a'}\sum_{k=1}^{a-1}q^{\overline{a}-\overline{k}}q^{\delta_{i\,k'}-\delta_{i\,k}}\epsilon_k\epsilon_as_{j\,k'}s_{i\,k}\\
        &-(q-q^{-1})^2\delta_{b\,a'}\sum_{k=i'+1}^{a-1}q^{\overline{a}-\overline{k}}\epsilon_k\epsilon_as_{j\,i}s_{k'\,k}\\
        &-(q-q^{-1})\delta_{b\,i'}\sum_{k=i'+1}q^{\overline{i}-\overline{k'}}q^{\delta_{a\,i'}-\delta_{a\,i}}\epsilon_i\epsilon_{k'}s_{j\,k}s_{k'\,a}\\
        &+(q-q^{-1})\delta_{a\,j'}\sum_{k=j'+1}q^{\overline{k}-\overline{j'}}q^{\delta_{i\,j}-\delta_{i\,j'}}\epsilon_k\epsilon_{j'}s_{i\,k}s_{k'\,b}\\
        &+(q-q^{-1})\delta_{i\,j'}\sum_{k=i+1}q^{\overline{k}-\overline{i}}q^{\delta_{a\,k'}-\delta_{a\,k}}\epsilon_i\epsilon_{k}s_{k\,a}s_{k'\,b}\\
        &+(q-q^{-1})^2\delta_{i\,j'}\sum_{k=i+1}^{a'-1}q^{\overline{k}-\overline{i}}\epsilon_k\epsilon_{i}s_{k\,k'}s_{a\,b}.\\
    \end{aligned}    
\end{equation}

\begin{theorem}\label{pbw}
Consider the twisted quantized eveloping algebra $U_q^{tw}(\mathfrak{gl}_n)$ with generators $s_{i\,j}$ for $i,j = 1\cdots,n$ and $i\geq j$, the defining relations are written in terms of the matrix $S=(s_{i\,j})$ by the relations (\ref{sij})--(\ref{Scentre}). Then the ordered monomials of the form
        \begin{equation*}
s_{2\,1}^{k_{2\,1}}s_{3\,1}^{k_{3\,1}}s_{3\,2}^{k_{3\,2}}\cdots s_{n\,1}^{k_{n\,1}}\cdots s_{n,n-1}^{k_{n,n-1}}s_{n+1\,1}^{k_{n+1\,1}}\cdots s_{n+1\,n}^{k_{n+1\,n}}\cdots s_{2n\,1}^{k_{2n\,1}},
    \end{equation*}
with nonnegative powers $k_{i\,j}$ constitute a basis of the algebra $U_q^{tw}(\mathfrak{gl}_n)$.
\end{theorem}
The detailed proof of this main theorem will be presented in Section 4. 

\subsection{An equivalent presentation of $U_q^{tw}(\mathfrak{gl}_n)$}
Let $C=\text{diag}(c_1,\cdots,c_n,c_{n'},\cdots,c_{1'})$ such that 
$$c_1c_{1'}=c_2c_{2'}=\cdots =c_nc_{n'}=\lambda\neq 0.$$
Note that $C$ is an invertible matrix and $C^t=\lambda C^{-1},\ C^u = C$. In \cite{N2}, it was shown that $C$ satisfies the reflection equation $$RC_1R^uC_2=C_2R^uC_1R.$$ 
Therefore, we can define a family of subalgebras for $U_q(\mathfrak{sp}_{2n})$ parametrized by the matrices $C$. 
\begin{lemma} \label{TT'}
    The assignment 
    \begin{equation*}
    \tau:    T\mapsto T'=\lambda^{-\frac{1}{2}}C^{\frac{1}{2}}T C^{\frac{1}{2}},\qquad \overline{T}\mapsto \overline{T}'=\lambda^{-\frac{1}{2}}C^{\frac{1}{2}}\overline{T} C^{\frac{1}{2}},
    \end{equation*}
    defines an algebraic automorphism of $U_q(\mathfrak{sp}_{2n})$.
\end{lemma}
\begin{proof}
    The only nontrivial part of this derivation is to prove the RTT relations (\ref{Rtt}) and the central relation (\ref{cent})  separately in the algebra $U_q(\mathfrak{sp}_{2n})$. It follows from the properties of $C$ and the specific form of $R$
    \begin{equation*}
        R=C_2^{-\frac{1}{2}}C_1^{-\frac{1}{2}}R C_1^{\frac{1}{2}}C_2^{\frac{1}{2}}=C_2^{\frac{1}{2}}C_1^{\frac{1}{2}}R C_1^{-\frac{1}{2}}C_2^{-\frac{1}{2}}.
    \end{equation*}
	  Then one has
	\begin{equation*}
		RC_1^{\frac{1}{2}}C_2^{\frac{1}{2}}\overline{T}_1T_2C_1^{\frac{1}{2}}C_2^{\frac{1}{2}}=C_2^{\frac{1}{2}}C_1^{\frac{1}{2}}T_2\overline{T}_1C_2^{\frac{1}{2}}C_1^{\frac{1}{2}}R,
	\end{equation*}
	and
	\begin{equation*}
		R(C_1^{\frac{1}{2}}\overline{T}_1C_1^{\frac{1}{2}})(C_2^{\frac{1}{2}}T_2C_2^{\frac{1}{2}})=(C_2^{\frac{1}{2}}T_2C_2^{\frac{1}{2}})(C_1^{\frac{1}{2}}\overline{T}_1C_1^{\frac{1}{2}})R,
	\end{equation*}
	which implies
	\begin{equation*}
		RT'_1T'_2=T'_2T'_1R,\qquad R\overline{T}'_1\overline{T}'_2=\overline{T}'_2\overline{T}'_1R,\qquad R\overline{T}'_1T'_2=T_2\overline{T}'_1R.
	\end{equation*}
	Other relations can be obtained similarly. From the equations 
	\begin{equation*}
		(C^{\frac{1}{2}})^t=(C^{\frac{1}{2}})^{-1}\lambda^{\frac{1}{2}},\quad (T')^t=\lambda^{\frac{1}{2}}(C^{\frac{1}{2}})^{-1}T^t(C^{\frac{1}{2}})^{-1},
	\end{equation*} 
	it follows that
	\begin{equation*}
		\begin{aligned}
			T'D(T')^tD^{-1}=&\lambda^{-\frac{1}{2}}C^{\frac{1}{2}}TC^{\frac{1}{2}}D\lambda^{\frac{1}{2}}(C^{\frac{1}{2}})^{-1}T^t(C^{\frac{1}{2}})^{-1}D^{-1}\\
			=&C^{\frac{1}{2}}TDT^t(C^{\frac{1}{2}})^{-1}D^{-1}=C^{\frac{1}{2}}D(C^{\frac{1}{2}})^{-1}D^{-1}=I.
		\end{aligned}
	\end{equation*}
	For $\overline{T}'D(\overline{T}')^tD^{-1}=I$, it is similar.  Thus, the proof is complete.
\end{proof}
Put $C=\sqrt{-1}J$. We define the matrix 
	$S'=J^{\frac{1}{2}}T\overline{T}^uJ^{\frac{1}{2}},$
 where $J^{\frac{1}{2}}=$diag $(\beta_1,\cdots\beta_{2n})$.
\begin{definition}
The new presentation of the quantized universal enveloping algebra, denoted by $U_q^{'tw}(\mathfrak{gl}_n)$, is the subalgebra of $U_q(\mathfrak{sp}_{2n})$ generated by the entries of matrix $S'$.
\end{definition}
\begin{proposition}\label{S'}
There exists an isomorphism from $U_q^{tw}(\mathfrak{gl}_n)$ to $U_q^{'tw}(\mathfrak{gl}_n)$.
\end{proposition}
\begin{proof}
The isomorphism follows from Lemma \ref{TT'} that $S'=\tau(T)J(\tau(T))^u$.
\end{proof}

Under the notation $S'=(s'_{i\,j})_{2n\times 2n}$,  we have $s_{i\,j}'=\beta_i\beta_j\sum_{k=j}^it_{i\,k}\overline{t}_{j\,k}$. Using the original notations of $S$, then we have
\begin{equation*}
    \begin{aligned}
s_{i+1\,i}&=t_{i+1\,i}\overline{t}_{i\,i}+t_{i+1\,i+1}\overline{t}_{i\,i+1},\ 1\leq i\leq n-1,\\ s_{n+1\,n}&=t_{n+1\,n}\overline{t}_{n\,n}-t_{n+1\,n + 1}\overline{t}_{n\,n+1},
    \end{aligned}
\end{equation*}
while for $S'$, one has
\begin{equation*}
    \begin{aligned}
        s'_{i+1\,i}&=\beta_i\beta_{i+1}(t_{i+1\,i}\overline{t}_{i\,i}+t_{i + 1\,i+1}\overline{t}_{i\,i+1}),\ 1\leq i\leq n-1,\\ 
        s'_{n+1\,n}&=\beta_n\beta_{n+1}(t_{n+1\,n}\overline{t}_{n\,n}+t_{n+1\,n+1}\overline{t}_{n\,n+1}).
    \end{aligned}
\end{equation*}
The new presentation $U_q^{'tw}(\mathfrak{gl}_n)$ provides the foundation for establishing an isomorphism between $U_q^{'tw}(\mathfrak{gl}_n)$ and $\imath$quantum group $\mathcal{U}^{\imath}$, as discussed in the following text.
\begin{remark}
M. Noumi provided a list of constant solutions, denoted by $C$,  of the reflection equation for the corresponding $R$-matrix, which was used to construct a coideal algebra corresponding to the symmetric pairs in \cite{N2}. Moreover, it was shown in \cite{M1} that a family of subalgebras in $U$, parametrized by $C$, are isomorphic to each other as abstract algebras under the transformation $S\to CSC$ in type $A$. 
\end{remark}

\section{The proof of the theorem \ref{pbw}}
In this section, we will first prove a weak form of the PBW theorem \ref{pbw}, which asserts that under the lexicographic order, the ordered monomials of the form 
\begin{equation*}
\prod_{i=1,\cdots,2n}^{\to}s_{i\,1}^{k_{i\,1}}s_{i\,2}^{k_{i\,2}}\cdots s_{i\,i'}^{k_{i\,i'}}
\end{equation*}
with nonnegative powers $k_{i\,j}$ linearly span the abstract associative algebra $\mathcal{S}$ with $2n^2 + n$ generators $s_{i\,j}$, $i\geq j$, $i,j = 1,\cdots,2n$.
The defining relations can be expressed in terms of the matrix
$S=(s_{i\,j})$ by the relation (\ref{sij})--(\ref{Scentre}).
It should be noted that in this section, we denote the matrix 
 $TJ\overline{T}^u$ by $\widetilde{S}$ and refer to its matrix elements as $\widetilde{s}_{i\,j}$ exclusively for this context.
 Let 
 \begin{align*}
 \Omega_1&=\{s_{k_1\,l_1}s_{k_2\,l_2}\cdots s_{k_p\,l_p},\ l_m'\geq k_m > l_m,\ m = 1,\cdots,p\},\\
 \Omega_2&=\{s_{i_1\,j_1}s_{i_2\,j_2}\cdots s_{i_p\,j_p},\ i_m > j_m > i_m',\ m = 1,\cdots,p\}
 \end{align*}
 denote the sets of monomials generated by $S$.
In fact, although $\mathcal{S}$ has $2n^2 + n$ generators, 
we will prove in Lemma \ref{sijspan}  that the generators
$s_{i\,j}\in \Omega_2$ can be spanned by ordered monomials in $\Omega_1$.
Furthermore, since $s_{i\,i}$ are all scalar, 
the algebra $\mathcal{S}$ effectively has only
$n^2$ generators, namely $s_{k\,l}$ for $l'\geq k > l$. 

\begin{lemma} \label{sijspan}
The generator $s_{i\,j}\in\Omega_2$ can be expressed as a linear combination of ordered monomials of $\Omega_1$ of the form 
\begin{equation*}
    s_{i_1\,a_1} s_{i_2\,a_2}\cdots s_{i_p\,a_p},
\end{equation*}
    where $i'< i_1 < i_2 < \cdots < i_p\leq j'$.
\end{lemma}
\begin{proof}
Due to (\ref{scenspecific2}), for $i > j > i'$, $s_{i\,j}\in\Omega_2$, it follows that 
\begin{equation}\label{sijs'}
    q^{\overline{j}-\overline{i}}s_{i\,j}=qs_{j'\,i'}+\sum_{k = j + 1}^{i - 1}q^{\overline{j}-\overline{k}+1}\epsilon_{k}s_{k'\,i'}s_{k\,j}.
\end{equation}
It is obvious that $s_{j'\,i'}, s_{k'\,i'}\in\Omega_1$. 
Let $m = i - j$,  we proceed by the induction on $m$ to prove this lemma.
If $m = 1$, we have $q^{\overline{j}-\overline{i}}s_{i\,j}=qs_{j'\,i'}$, so the assertion holds.
Suppose that the claim holds for $m < l$,, where $l$ is some integer.
In equation (\ref{sijs'}), $s_{k\,j}\in \Omega_2 $ when $k > j'$, and  $s_{k\,j}\in \Omega_1$ when $k\leq j'$. Therefore, we need to consider two cases.\\
{\bf Case 1:} $s_{k\,j}\in \Omega_2$, since $k - j < i - j = l$, $s_{k\,j}$ can be spanned by the ordered monomials $s_{k_1\,j_1}\cdots s_{k_p\,j_p}$ of $\Omega_1$ with $k'<k_1<\cdots k_p\leq j'$ by assumption.
So, $s_{k'\,i'}s_{k\,j}$ can be spanned by the monomials $s_{k'\,i'}s_{k_1\,j_1}\cdots s_{k_p\,j_p}$ of $\Omega_1$ where $i'<k'<k_1<\cdots k_p\leq j'$.\\
{\bf Case 2:} When $k\leq j'$, $s_{k\,j}\in \Omega_1$.
If $k > n$, then  $k'<k$ and the assertion holds.
If $k\leq n$, we have $k < k'$.
It follows from (\ref{sre}) that
\begin{align*}
s_{k'\,i'}s_{k\,j}=qs_{k\,j}s_{k'\,i'}+(q-q^{-1})\sum_{t'=k'+ 1}^{j'}q^{\overline{t}+\overline{k}+1}s_{t'\,i'}s_{t\,j}. 
\end{align*}
Hence, $s_{k'\,i'}s_{k\,j}$ 
  can be expressed as a linear combination of monomials $s_{t_p\,j}s_{t'_p\,i'}$, where $j < t_p < t_p'<j$, by induction on $t'$.
Note that when $j\geq n$, the second case does not occur.
\end{proof}

\begin{lemma} \label{weakp}
			The ordered monomials of the form 
			\begin{equation} \label{sia1}
			s_{i_1\,a_1}^{k_1} s_{i_2\,a_2}^{k_2}\cdots s_{i_p\,a_p}^{k_p},
			\end{equation}
			where $k_{i}$ are nonnegative integers, linearly span the abstract associative algebra $\mathcal{S}$. Here $a_m'\geq  i_m> a_m$ for $m=1,\cdots,p$, and the indices satisfy $i_1\leq i_2\leq \cdots \leq i_p$.
\end{lemma}
\begin{proof} 
Given a monomial 
	\begin{equation} \label{monos}
		s_{j_1\,b_1}\cdots s_{j_p\,b_p},
	\end{equation}
we define its length as $p$ and its weight $\omega$ as $\omega=j_1+\cdots + j_p$.
For a monomial of length $p$, the weight can range from $p$ to $2pn$. 
We will proceed by induction on $\omega$.
Directly proving that all elements $s_{i\,j}$, $i\geq j$, can be linearly spanned by ordered monomials would be quite challenging. 
However, in Lemma \ref{sijspan}, we have already established  that the generating matrix can be written as $S=(s_{i\,j}),$ where $j'\geq i>j$.
Therefore, we restrict our attention to the case where the monomials belong to $\Omega_1$, and we will prove their ordering.
From equation (\ref{sre}), we obtain the following equality modulo products of weight $\omega$ less than $i+j$, for $i >j,\ a'\geq i> a,\ b'\geq j> b$,
	\begin{equation} \label{sijm}
        \begin{aligned} 
            q^{\delta_{i\,j}-\delta_{i\,j'}+\delta_{a\,j}-\delta_{a\,j'}}s_{i\,a}s_{j\,b}
            \equiv \ &q^{\delta_{i\,b}-\delta_{i\,b'}+\delta_{a\,b}-\delta_{a\,b'}}s_{j\,b}s_{i\,a}\\
            &+(q - q^{-1})q^{-\delta_{i\,a'}}\delta_{b < a}s_{j\,a}s_{i\,b}\\
            &-(q - q^{-1})q^{\delta_{i\,j}-\delta_{i\,j'}}\delta_{j > a,i'}s_{i\,j}s_{a\,b}\\
            &-(q - q^{-1})^2\delta_{a\,i'}\sum_{k = i'+1}^{b',j}q^{\overline{k}-\overline{i'}}\epsilon_i\epsilon_{k'}\delta_{b < i'}\delta_{j > k'}s_{j\,k}s_{k'\,b}\\
            &-(q - q^{-1})\delta_{b\,i'}\sum_{k = i'+1}^{j,a'}q^{\overline{i}-\overline{k'}}q^{\delta_{a\,i'}-\delta_{a\,i}}\epsilon_i\epsilon_{k'}\delta_{j > k'}s_{j\,k}s_{k'\,a}\\
            &+(q - q^{-1})\delta_{i\,j'}\sum_{k = i + 1}^{b'}q^{\overline{k}-\overline{i}}q^{\delta_{a\,k'}-\delta_{a\,k}}\epsilon_i\epsilon_{k}s_{k\,a}s_{k'\,b}.\\
        \end{aligned} 
\end{equation}
While the weight $\omega$ of some monomials on the right-hand side is less than $i+j$, these monomials belong to $\Omega_2$. 
 When we express these monomials in terms of monomials from $\Omega_1$ using Lemma \ref{sijspan}, their weights will exceed $i+j$.
 
 From equation (\ref{sre}) and lemma \ref{sijspan}, with $b<a\leq i <j,j'<i'\leq a'<b'$, we have 
	\begin{equation*}
    \begin{aligned}
        q^{-\delta_{i\,a'}}s_{i\,j}s_{a\,b}&= s_{a\,b}s_{i\,j}+(q - q^{-1})\delta_{i\,a'}\sum_{k = i + 1}^{b'}s_{k\,j}s_{k'\,b}\\
        &=\operatorname{span}_{\C(q)}\{s_{a\,b}s_{i_1\,j_1} s_{i_2\,j_2}\cdots s_{i_p\,j_p}\}+\delta_{i\,a'}\sum_{k = i + 1}^{b'}s_{k\,j}s_{k'\,b}.
    \end{aligned}
\end{equation*}
It is evident that the monomial $s_{a\,b}s_{i_1\,j_1} s_{i_2\,j_2}\cdots s_{i_p\,j_p}$ has the ordering $a\leq i'< i_1< i_2< \cdots <i_p\leq j'$.
The item $\delta_{i\,a'}\sum_{k=i+1}^{b'}s_{k\,j}s_{k'\,b}$ can be expressed as a linear combination of ordered monomials in $\Omega_1$ by induction on $k$, based on the equation in (\ref{sre}) given by
\begin{align*}
q^{-1}s_{k\,j}s_{k'\,b}=s_{k'\,b}s_{k\,j}+(q-q^{-1})\sum_{t=k+1}^{b'}q^{\overline{t}-\overline{k}}s_{t\,j}s_{t'\,b}.
\end{align*}

For $\delta_{a\,i'}\sum_{k=i'+1}\delta_{b<i'}\delta_{j>k'}s_{j\,k}s_{k'\,b}\in \Omega_2$, it follows from Lemma \ref{sijspan} that
		\begin{equation*}
			\begin{aligned}
				s_{j\,k}s_{k'\,b}= \operatorname{span}_{\C(q)}\{s_{j_1\,k_1} s_{j_2\,k_2}\cdots s_{j_p\,k_p}s_{k'\,b}\},
			\end{aligned}
		\end{equation*}
		where  $j'< j_1< j_2< \cdots < j_p\leq k'$. 
        
       A similar argument applies to the case of $\delta_{b\,i'}\sum_{k=i'+1}^{j\,a'}q^{\overline{i}-\overline{k'}}q^{\delta_{a\,i'}-\delta_{a\,i}}\epsilon_i\epsilon_{k'}\delta_{j>k'}s_{j\,k}s_{k'\,a}$.
       
For	$\delta_{i\,j'}\sum_{k=i+1}^{b'}q^{\overline{k}+\delta_{a\,k'}}\epsilon_{k}s_{k\,a}s_{k'\,b}$, we consider two cases.\\
{\bf Case 1:} If $a\leq b$, we have $a\leq b\leq k'<j=i'<i=j'<k<b'\leq a'$, which implies $s_{k\,a}s_{k'\,b}\in\Omega_1$. From equation (\ref{sre}), we deduce
\begin{equation*}
    q^{-1}s_{k\,a}s_{k'\,b}\equiv q^{-\delta_{k\,b'}+\delta{a\,b}}s_{k'\,b}s_{k\,a}+(q - q^{-1})\sum_{t = k + 1}^{b'}q^{\overline{t}-\overline{k}}q^{\delta_{a\,t'}-\delta_{a\,t}}s_{t\,a}s_{t'\,b}.
\end{equation*}
This case can be verified by induction on $k$.\\
{\bf Case 2:} If $a> b$, we deduce $b<a\leq j=i'<i=j'\leq a'<b'$. 
Since $i<k\leq b'$ for $k>a'$, we have $s_{k\,a}\in \Omega_2$.
Rewriting the central relation (\ref{nScen}) yields
		\begin{equation} \label{olss}
        (\sum_{t=a}^{i}+\sum_{t=i+1}^{b'})q^{\overline{t}}\overline{s}_{a'\,t'}s_{t'\,b}=0.
		\end{equation}
By Lemma \ref{overss}, equation (\ref{olss}) becomes
		\begin{equation}\label{recen}
			(\sum_{t=a+1}^{i}+\sum_{t=i+1,t\neq a'}^{b'})q^{\overline{t}}\epsilon_{a}\epsilon_{t}s_{t\,a}s_{t'\,b}+q^{-1-\overline{a}}\overline{s}_{a'\,a}s_{a\,b}+q^{-1+\overline{a}}\epsilon_{a'}s_{a'\,b}=0.
		\end{equation}
Then, equation (\ref{recen}) leads to 
		\begin{equation*}
			\begin{aligned}
\sum_{k=i+1}^{b'}q^{\overline{k}+\delta_{a\,k'}}\epsilon_{k}s_{k\,a}s_{k'\,b}
=&q^{\overline{a'}+1}s_{a'\,a}s_{a\,b}-\sum_{k=a+1}^{i}q^{\overline{k}}\epsilon_{a}\epsilon_{k}s_{k\,a}s_{k'\,b}-q^{-1-\overline{a}}\overline{s}_{a'\,a}s_{a\,b}-q^{-1+\overline{a}}\epsilon_{a'}s_{a'\,b}\\
=&q^{\overline{a'}+1}s_{a'\,a}s_{a\,b}+\sum_{k=a+1}^{i}q^{\overline{k}}s_{k\,a}s_{k'\,b}-q^{-1+\overline{a}}\epsilon_{a'}s_{a'\,b}-q^{1-\overline{a}}s_{a'\,a}s_{a\,b}\\
&-\sum_{m=a+1}^{n}(-1)^{k_m}q^{k_m}s_{m'\,m}s_{a\,b}\equiv\sum_{k=a+1}^{i}q^{\overline{k}}s_{k\,a}s_{k'\,b}-q^{-1+\overline{a}}\epsilon_{a'}s_{a'\,b}.
			\end{aligned}
		\end{equation*}
It follows that $\sum_{k=a+1}^{i-1}q^{\overline{k}}\epsilon_{a}\epsilon_{k}s_{k\,a}s_{k'\,b}\in\Omega_1$.
By induction on $k$, as in the previous case, the assertion holds.
The proof allows us to express the terms in equation (\ref{monos}), which belong to $\Omega_1$, modulo monomials of weight $\omega$ less than $i+j$, as a linear combination of monomials $s_{i_1\,a_1}s_{i_2\,a_2}\cdots s_{i_p\,a_p}\in \Omega_1$ such that $i_1\leq i_2\leq\cdots\leq i_p$.
		\end{proof}
	
\begin{lemma}\label{lem:pbw}
	Consider the abstract associative algebra $\mathcal{S}$ with $n^2$ generators $s_{i\,j}$, where $i,j=1,\cdots,2n$, $j\geq i> j'$.
    The algebra is linearly spanned by the ordered monomials of the form 
		\begin{equation*}
\prod_{i=1,\cdots,2n}^{\to}s_{i\,1}^{k_{i\,1}}s_{i\,2}^{k_{i\,2}}\cdots s_{i\,i'}^{k_{i\,i'}},
		\end{equation*}
		where the powers $k_{i\,j}$ are nonnegative integers. 
	\end{lemma}
    \begin{proof}

Following the proof of Lemma \ref{weakp}, we consider a submonomial $s_{i\,c_1}s_{i\,c_1}\cdots s_{i\,c_r}$ from equation (\ref{sia1}), which contains generators with the same first index. 
By equation (\ref{sijm}), we obtain the following equality, modulo products of weight less than $2i$, which belongs to $\Omega_1$ for $a'\geq i > a > b,$
    \begin{equation*}
        \begin{aligned}
            q^{1+\delta_{a\,i}-\delta_{a\,i'}}s_{i\,a}s_{i\,b}
            \equiv \  &q^{\delta_{i\,b}-\delta_{i\,b'}-\delta_{a\,b'}}s_{i\,b}s_{i\,a}
            +(q - q^{-1})q^{-\delta_{i\,a'}}s_{i\,a}s_{i\,b}
            \\
            &-(q - q^{-1})^2\delta_{a\,i'}\sum_{k = i'+1}^{i,b'}q^{\overline{k}-\overline{i'}}\epsilon_i\epsilon_{k'}\delta_{i > k'}s_{i\,k}s_{k'\,b}.
        \end{aligned}    
    \end{equation*}
When $i\neq a'$, the lemma obviously holds.
Note that when $\delta_{i > k'}$, one has $s_{i\,k}\in\Omega_2$.
Then, for $i = a'$, it follows that
\begin{equation*}
        \begin{aligned}
q^{-2}s_{i\,a}s_{i\,b}\equiv s_{i\,b}s_{i\,a}-(q - q^{-1})^2\sum_{k = i'+1}^{i,b'}q^{\overline{k}-\overline{i'}}\epsilon_i\epsilon_{k'}\delta_{i > k'}s_{i\,k}s_{k'\,b}.
    \end{aligned}    
    \end{equation*}
From lemma \ref{sijspan}, we deduce
    \begin{equation*}
        s_{i\,k}s_{k'\,b}=\operatorname{span}_{\C(q)}\{ s_{i_1\,k_1} s_{i_2\,k_2}\cdots s_{i_p\,k_p}s_{k'\,b}\},
    \end{equation*}
    where $i'< i_1< i_2<\cdots < i_p\leq k'$. 
Hence, we have discussed the case when $i_p = k'$ and $a_p > b$.
Owing to equation (\ref{sre}), it follows that 
    $$q^{-1}s_{k'\,a_p}s_{k'\,b}= s_{k'\,b}s_{k'\,a_p}+(q - q^{-1})q\epsilon_{k}s_{a_p\,b}.$$
By repeatedly using this relation, we can transform the submonomial to the required form.
\end{proof}

Having established the weak PBW theorem above, we now proceed to prove the linear independence of the monomials. 
Before that, we recall the $\mathbb{A}$-form of the quantum algebra.

Consider $q$ as a formal variable and $U_q(\mathfrak{sp}_{2n})$ as an algebra over $\C(q)$. Set $\mathbb{A}=\C[q,q^{-1}]$. Denoted $U_{\mathbb{A}}$ as the $\mathbb{A}$-subalgebra of $U_q(\mathfrak{sp}_{2n})$ generated by the elements $\tau_{i\,j}$ and $\overline{\tau}_{i\,j}$, which are defined by 
\begin{equation}\label{tau1}
    \tau_{i\,j} =\dfrac{t_{i\,j}}{q_i - q_{i}^{-1}}\quad \text{for}\quad i > j,\qquad  \overline{\tau}_{i\,j} =\dfrac{\overline{t}_{i\,j}}{q_i - q_{i}^{-1}}\quad \text{for}\quad i < j,
\end{equation}
and
\begin{equation}\label{tau2}
    \tau_{i\,i} =\dfrac{t_{i\,i}-1}{q - 1},\qquad \overline{\tau}_{i\,i} =\dfrac{\overline{t}_{i\,i}-1}{q - 1}.
\end{equation}
Then, there exits a natural isomorphism 
$$U_{\mathbb{A}}\otimes_{\mathbb{A}}\C \cong U(\mathfrak{sp}_{2n}),$$
where the action of $\mathbb{A}$ on $\C$ is defined via the evaluation $q = 1$, as discussed in [\citenum{VC}, Sec 9.2]. Note that $\tau_{i\,j}$ and $\overline{\tau}_{i\,j}$  specialize to the elements $F_{i\,j}$ and $-F_{i\,j}$ respectively, which are the generators of the symplectic Lie algebra $\mathfrak{sp}_{2n}$.
Given a subalgebra $V$ of $ U_q(\mathfrak{sp}_{2n})$, we define $V_{\mathbb{A}}=V\cap U_{\mathbb{A}}$. 
Following  [\citenum{L2}, Sec 1], we claim that $V$ specializes to a subalgebra $\overline{V}$ of $ U(\mathfrak{sp}_{2n})$ if $V_{\mathbb{A}}\otimes_{\mathbb{A}}\C \cong \overline{V}$. 
The method for proving the PBW theorem of twisted quantized enveloping algebra was introduced by A. Molev et al. in \cite{M1}.

\begin{theorem}\label{UPbw}
The abstract associative algebra $\mathcal{S}$, generated by $n^2$ elements $s_{i\,j},\ i,j = 1,\cdots, 2n$, $j'\geq i > j$, with defining relations (\ref{sij})--(\ref{Scentre}), is isomorphic to $U_q^{tw}(\mathfrak{gl}_n)$. The monomials 
\begin{equation*}
s_{2\,1}^{k_{2\,1}}s_{3\,1}^{k_{3\,1}}s_{3\,2}^{k_{3\,2}}\cdots s_{n\,1}^{k_{n\,1}}\cdots s_{n,n - 1}^{k_{n,n - 1}}s_{n + 1\,1}^{k_{n + 1\,1}}\cdots s_{n + 1\,n}^{k_{n + 1\,n}}\cdots s_{2n\,1}^{k_{2n\,1}},
\end{equation*}
with nonnegative powers $k_{i\,j}$, form a basis of the algebra $\mathcal{S}$.
\end{theorem}
\begin{proof}
$S$ is the generating matrix of the abstract algebra $\mathcal{S}$ described above, and let
$\widetilde{S}=TJ\overline{T}^u$ be the generating matrix of $ U_q^{tw}(\mathfrak{gl}_n)$.
There exists a natural map from the matrix $S\mapsto \widetilde{S}$, which induces an algebra homomorphism $\rho :\mathcal{S}\to U_q^{tw}(\mathfrak{gl}_n)$. 
It is clear that this homomorphism is surjective. Thus, it suffices to prove that the homomorphism is injective.
According to Lemma \ref{lem:pbw}, the monomials 
\begin{equation*}    
\prod_{i = 1,\cdots,2n}^{\to}s_{i\,1}^{k_{i\,1}}s_{i\,2}^{k_{i\,2}}\cdots s_{i\,i'}^{k_{i\,i'}}
\end{equation*}
span the abstract algebra $\mathcal{S}$. 
We will show that the images of these monomials under $\rho$ are linearly independent.
Regarding $q$ as a formal variable, we denote by $\mathbb{A}$  the algebra of Laurent polynomials $\C [q,q^{-1}]$.
Let $V = U_q^{tw}(\mathfrak{gl}_n)$ and note that the subalgebra $V_{\mathbb{A}}$ is generated by the elements $\sigma_{i\,j}:=\widetilde{s}_{i\,j}/(q - q^{-1})$,  where $j'\geq i > j$. 
From equations (\ref{tau1}) and (\ref{tau2}), we deduce that 
\begin{equation*}\label{basis}
\sigma_{i\,j}=\tau_{i\,j}+\overline{\tau}_{j\,i}+(q - 1)(\tau_{i\,i}\overline{\tau}_{j\,i}+\tau_{i\,j}\overline{\tau}_{j\,j})+(q - q^{-1})\sum_{j < a < i}\tau_{i\,a}\overline{\tau}_{j\,a}.
\end{equation*}
Therefore, the image of $\sigma_{i\,j}$ in $V_\mathbb{A}\otimes_\mathbb{A}\C$ under $\rho$ is given by $G_{i\,j}:=\epsilon_{i}F_{i\,j}-\epsilon_{j}F_{j\,i}$. 
Note that the elements $G_{i\,j}$ where $j'\geq i > j$, form a basis of a subalgebra of $\mathfrak{sp}_{2n}$ that is isomorphic to the general Lie algebra $\mathfrak{gl}_n$, as shown in Proposition \ref{propgl}.
Since the PBW theorem for $\mathfrak{gl}_n$ is well known, the ordered monomials in $G_{i\,j}$ are linearly independent. 
Now assume that there exists a nontrivial linear combination of the ordered monomials in $\sigma_{i\,j}$ that equals zero:
\begin{equation}\label{lin}
    \sum_{(k)}c_{(k)}\prod_{j'\geq i > j}\sigma_{i\,j}^{k_{i\,j}}=0,
\end{equation}
where $(k)=(k_{i\,j}|j'\geq i > j)$, and $c_{(k)}\in\mathbb{A}$.
We further assume that there is at least one coefficient $c_{(k)}$ is nonzero when evaluated at $q = 1$. 
Taking the image of the above equation (\ref{lin}) in $V_\mathbb{A}\otimes_\mathbb{A}\C$ under the mapping $\rho$, we obtain a nontrivial linear combination of the ordered monomials in $G_{i\,j}$ that equal to zero. This leads to a contradiction, thereby proving the claim.
\end{proof}
Based on the above proof, we have obtained the proof of Theorem \ref{pbw}.

\section{Connection with $\imath$quantum group}
In \cite{L1}, G. Letzter provided the general generator for the $\imath$quantum group.
Let $\{\alpha_i | i\in I\}$ be a set of simple roots for the root system of $\mathfrak{g}$.
S. Kolb and J. Pellegrini considered three classes of example in \cite{KP11}.
When the involution $\theta$ coincides with the Chevalley automorphism, 
the $\imath$quantum group is generated by $B_i=f_i-k_i^{-1}e_i$, for all $i\in I$.
When $\mathfrak{g}=\mathfrak{sp}_{2n}$, there exists a natural isomorphism between $U_q^{tw}(\mathfrak{gl}_n)$ and the $\imath$quantum group $\mathcal{U}^{\imath}$ in this case, which will be proved in Theorem \ref{sisoui}.

In this section, we aim to derive the PBW basis for the
$\imath$quantum group  $\mathcal{U}^{\imath}$ for type CI, defined in \cite{KP11}. This will be achieved through the connection between the twisted quantized algebra and  
$\imath$quantum group.

\begin{definition}([\citenum{KP11}, Proposition 3.1])
The $\imath$quantum group $\mathcal{U}^{\imath}$ of type CI is an associative algebra over $\C(q)$, generated by the elements $\{B_i \ | \  i\leq n\}$, satisfying the following relations
    \begin{align}
    &B_iB_j-B_jB_i=0,&\quad &c_{ij}=0,\\
    &\sum_{s=0}^{2}(-1)^s\begin{bmatrix} 2 \\ s \end{bmatrix}_{q_i}B_i^{2-s}B_jB_i^s=-q^{-1}_iB_j,&\quad&c_{ij}=-1,\\
    &\sum_{s=0}^{3}(-1)^s\begin{bmatrix} 3 \\ s \end{bmatrix}_{q_i}B_{n-1}^{3-s}B_{n}B_{n-1}^s=-q^{-1}[2]^2(B_{n-1}B_{n}-B_{n}B_{n-1}).&\quad& \ \label{sn-1sn}
    \end{align}
\end{definition}

\begin{theorem}\label{sisoui}
There exists a natural isomorphism $\widetilde{\phi}$ between $U_q^{tw}(\mathfrak{gl}_n)$ and the $\imath$quantum group $\mathcal{U}^{\imath}$ of type CI, such that $\widetilde{\phi}(s_{i+1\, i}')=(q_i-q_i^{-1})B_i$, where $s_{i+1\, i}'=\beta_i\beta_{i+1}(t_{i+1\, i}\overline{t}_{ii}+t_{i+1\, i+1}\overline{t}_{i\, i+1})$, $i\leq n$, as defined in Proposition \ref{S'}.
\end{theorem}

\begin{proof}
Since $U_q^{tw}(\mathfrak{gl}_n)$ and $\mathcal{U}^{\imath}$ are  subalgebras of $U_q(\mathfrak{sp}_{2n})$ and $\mathcal{U}_q(\mathfrak{sp}_{2n})$, respectively,
together with the isomorphism $\phi$ from  $U_q(\mathfrak{sp}_{2n})$ to $\mathcal{U}_q(\mathfrak{sp}_{2n})$, as discussed in Theorem \ref{iso}, we have the following commutative diagram:
		\begin{center}
	\begin{tikzcd}
		U_q^{tw}(\mathfrak{gl}_n)
		\arrow{r}\arrow{d}{\widetilde{\phi}} & U_q(\mathfrak{sp}_{2n}) \arrow{d}{\phi} \\
		\mathcal{U}^{\imath} \arrow{r} & \mathcal{U}_q(\mathfrak{sp}_{2n}),
	\end{tikzcd}
\end{center}
		where $\widetilde{\phi}=\phi|_{U_q^{tw}(\mathfrak{gl}_n)}$. 
        As a consequence, it follows that $\phi$ is both surjective and injective.
	\end{proof}
\begin{remark}
In [\citenum{L1}, Section 6.], G. Letzter shown that the coideal subalgebras $U_q^{tw}(\mathfrak{g})$, constructed via solutions of the reflection equation in \cite{N2}, are contained in Letzter's coideal algebra $\mathcal{U}^{\imath}$. 
These subalgebras may not, in general, be isomorphic. In this work, we establish an isomorphism between them for type 
$\operatorname{CI}$.
\end{remark}

The following relations hold in $U_q^{tw}(\mathfrak{gl}_n)$, which is an associative algebra with $n^2$ generators $s'_{i\,j}$, where $j'\geq i> j$.
\begin{lemma}\label{fres}
We have the following relations in generating series in $U_q^{tw}(\mathfrak{gl}_n)$,
		\begin{align}
&s'_{i\,i-1}=s'_{i'+1\,i'}, &\quad i\leq n,\label{s'ii-1}\\
&(q-q^{-1})s'_{i+1\,i-1}=qs'_{i+1\, i}s'_{i\,i-1}-s'_{i\,i-1}s'_{i+1\, i}, &\quad i< n,\label{s'i+1i-1}\\
			&(q^2-q^{-2})s'_{n+1\,i}=q^2s'_{n+1\,n}s'_{n\,i}-s'_{n\,i}s'_{n+1\,n}, &\quad i\leq n-1,\label{s'n+1n}\\
			&(q-q^{-1})s'_{i'+1\,i-1}=q^{-1}s'_{i'\,i-1}s'_{i\,i-1}-qs'_{i\,i-1}s'_{i'\,i-1}+(q-q^{-1})q^{-1}s'_{i'\,i},&\quad i\leq n,\label{s'i+1i-1}\\
			&(q-q^{-1})s'_{i\,j}=qs'_{i\,i-1}s'_{i-1\,j}-s'_{i-1\,j}s'_{i\,i-1}, &\text{otherwise}.\label{sijgen}
		\end{align}
	\end{lemma}

\begin{proof}
The relations (\ref{s'ii-1})--(\ref{s'i+1i-1}) are evident from (\ref{sre}). Relation (\ref{sijgen}) is preserved by induction on $k=i-j$.   
\end{proof}

Let $B_{i\,j}=(q-q^{-1})s_{i\,j}$, where $i\neq n+1,j\neq n$, and $B_{n+1\,n}=(q^2-q^{-2})s_{n+1\,n}$, the following lemma holds.
\begin{lemma}
It follows that the following relations involving $B_{i\,j}$:
	\begin{align*}
    &B_{i\,j}=qB_{i\,i-1}B_{i-1\,j}-B_{i-1\,j}B_{i\,i-1}, \quad |i-j|>1, i<n+1, i\neq j',\\
    &B_{n+1\,j}=q^2B_{n+1\,n}B_{n\,j}-B_{n\,j}B_{n+1\,n},\quad j\leq n-1,\\
    &B_{i\,j}=B_{i-1\,j}B_{i'+1\,i'}-qB_{i'+1\,i'}B_{i-1\,j}, \quad |i-j|> 1, i> n+1,\\
    &B_{i'+1\,i-1}=q^{-1}B_{i'\,i-1}B_{i\,i-1}-qB_{i\,i-1}B_{i'\,i-1}+q^{-1}B_{i'i},\quad i\neq n,\\
    &B_{n+2\,n-1}=q^{-1}B_{n+1\,n-1}B_{n\,n-1}-qB_{n\,n-1}B_{n+1\,n-1}+(1+q^{-2})B_{n+1\,n}.\\
		\end{align*}
\end{lemma}
\begin{proof}
It can be obtained directly by Lemma \ref{fres}.
\end{proof}
\begin{theorem} \label{pbwui}
The monomials of the form
		\begin{equation*}
		B_{2\,1}^{k_{2\,1}}B_{3\,1}^{k_{3\,1}}B_{3\,2}^{k_{3\,2}}\cdots B_{n\,1}^{k_{n\,1}}\cdots B_{n\,n-1}^{k_{n\,n-1}}B_{n+1\,1}^{k_{n+1\,1}}\cdots B_{n+1\,n}^{k_{n+1\,n}}\cdots B_{2n\,1}^{k_{2n\,1}}
		\end{equation*}
			with nonnegative powers $k_{i\,j}$ constitute a basis of the algebra $U^{\imath}$.
	\end{theorem}
    \begin{proof}
According to the isomorphism from $U_q^{tw}(\mathfrak{gl}_n)$ and $\mathcal{U}^{\imath}$ proved in Theorem \ref{sisoui}, we have Theorem \ref{pbwui}.
    \end{proof}
    \begin{remark}
      In fact, the PBW basis of the $\imath$quantum group $U^{\imath}$ generated by $B_i$ was obtained in \cite{XY14}. Our approach also provides the PBW basis for $\mathcal{U}^{\imath}$ as an application of the connection between  $U_q^{tw}(\mathfrak{gl}_n)$ and $\mathcal{U}^{\imath}$. 
    \end{remark}
   
\section{An associated Poisson algebra}
In this section, we use $S=TJ\overline{T}^u$, $s_{i\,j}=\sum_{k=j}^{i}\epsilon_kt_{i\,k}\overline{t}_{j\,k}$, rather than $S'$ defined in Section 3.2 to avoid confusion in the calculation. Actually, we have $s_m=s'_m$ for $m=1,\cdots, n-1$, while $s_n=-\sqrt{-1}s'_n$.
We first introduce the $\mathbb{A}$-form of $U_q^{tw}(\mathfrak{gl}_n)$. Let $q$ be a formal variable and consider $U_q^{tw}(\mathfrak{gl}_n)$ as an algebra over $\C(q)$.
Define $\mathbb{A}=\C[q,q^{-1}]$. 
Let $U'_{\mathbb{A}}$ be the $\mathbb{A}$-subalgebra of  $U_q^{tw}(\mathfrak{gl}_n)$ generated by elements $s_{i\,j}$, $i > j$, $i,j = 1,\cdots 2n.$
In fact, the following isomorphism holds
\begin{align}
         U'_{\mathbb{A}}\otimes_{\mathbb{A}}\C \cong \mathcal{P}_n,\label{S-PIso}
\end{align}
where the action of $\mathbb{A}$ on $\C$ is defined by the evaluation $q=1$. 
Here, $\mathcal{P}_n$ denotes the algebra of polynomials in the variables $a_{i\,j}$ for $ i > j$, $i,j = 1,\cdots 2n.$
The elements $a_{i\,j}$ corresponding to the images of $s_{i\,j}$ under the isomorphism (\ref{S-PIso}).
The algebra $\mathcal{P}_n$ equipped with the Poisson bracket $\{\cdot.\cdot\}$ introduced in \cite{M08}, 
\begin{align*}
  \{a_{i\,j},a_{k\,l}\}=\dfrac{s_{i\,j}s_{k\,l}-s_{k\,l}s_{i\,j}}{1 - q}\Bigg{|}_{q = 1}.  
\end{align*}
%The extended twisted quantized enveloping algebra $\widetilde{U}$ defined in Definition \ref{defextu} is an associative algebra generated by elements $s_{i\,j}$, $i > j$ and $i,j = 1,\cdots 2n.$
%Denote by $\widetilde{U}'_{\mathbb{A}}$, the $\mathbb{A}$-subalgebra of $\widetilde{U}$.
%As $q\to 1$, the extended twisted quantized enveloping algebra $\widetilde{U}$ specializes to the extended Poisson algebra $\widetilde{\mathcal{P}}_{n}$ of polynomials in variables $a_{i\,j}$, such that $s_{i\,j}$ specializes to $a_{i\,j}$ for $i > j$.
    \begin{theorem}
The Poisson algebra $\mathcal{P}_{n}$ has a Poisson bracket defined by 
        \begin{equation} \label{poissonre}
            \begin{aligned}
                \{a_{i\,j},a_{k\,l}\}= & (\delta_{i\,k}-\delta_{i\,k'}+\delta_{j\,k}-\delta_{j\,k'}-\delta_{i\,l}+\delta_{i\,l'}-\delta_{j\,l}+\delta_{j\,l'})a_{i\,j}a_{k\,l}\\
                &-2(\delta_{l < j}-\delta_{i < k})a_{k\,j}a_{i\,l}-2\delta_{l < i}a_{k\,i}a_{l\,j}+2\delta_{j < k}a_{i\,k}a_{j\,l}\\
                &+2\delta_{l\,j'}\sum_{m = 1}^{j - 1}\epsilon_m\epsilon_ja_{k\,m'}a_{i\,m}+2\delta_{l\,i'}\sum_{m = i'+1}\epsilon_i\epsilon_{m'}a_{k\,m}a_{m'\,j}\\
                &-2\delta_{j\,k'}\sum_{m = k'+1}\epsilon_m\epsilon_{k'}a_{i\,m}a_{m'\,l}-2\delta_{i\,k'}\sum_{m = i + 1}\epsilon_i\epsilon_{m}a_{m\,j}a_{m'\,l}.   
            \end{aligned}
        \end{equation}
    \end{theorem}
    \begin{proof}
From the reflection relations (\ref{sre})  and the definition of the Poisson bracket, the explicit formulas for $\{a_{i\,j},a_{k\,l}\}$ can be derived directly.
    \end{proof}
    \begin{remark}
        We organize the variables $a_{i\,j}$ into a lower triangular matrix $A$ with ones on the diagonal
        \begin{equation*}
            A = \begin{pmatrix}
                1 & 0 & 0 & 0 & \cdots & 0 & 0 \\
                a_{2\,1} & 1 & 0 & 0 & \cdots & 0 & 0 \\
                a_{3\,1} & a_{3\,2} & 1 & 0 & \cdots & 0 & 0 \\
                \vdots & \vdots & \vdots & \vdots & \ddots & \vdots & \vdots \\
                a_{2n-1\,1} & a_{2n-1\,2} & a_{2n-1\,3} & a_{2n-1\,4} & \cdots & -1 & 0\\
                a_{2n\,1} & a_{2n\,2} & a_{2n\,3} & a_{2n\,4} & \cdots & a_{2n\,2n-1} & -1
            \end{pmatrix}.
        \end{equation*}
        The Poisson brackets of $\mathcal{P}_{n}$ 
        can be expressed in matrix form. As defined earlier, we adopt the standard notations:
                $$A_1=\sum_{i,j}a_{i\,j}\otimes e_{i\,j}\otimes I, \quad A_2=\sum_{i,j}a_{i\,j}\otimes I\otimes e_{i\,j},$$
        where $I$ is the identity matrix. 
    By the observation, let
        $$r=\dfrac{R-I\otimes I-I\otimes I^u}{q-1}\Bigg{|}_{q=1}.$$
 It follows from the $R$-matrix (\ref{R}) and the reflection equation (\ref{sre}) that
        
        $$\{A_1,A_2\}=[r,A_1A_2]+A_1r^uA_2-A_2r^uA_1,$$
        where 
        $$
        r=\sum_{i}e_{i\,i}\otimes e_{i\,i}-\sum_{i}e_{i\,i}\otimes e_{i'\,i'}+2\sum_{i<j}e_{i\,j}\otimes e_{j\,i}-2\sum_{i<j}\epsilon_i\epsilon_je_{i\,j}\otimes e_{i'\,j'},
        $$
        
        $$
        r^u=\sum_{i}e_{i\,i}\otimes e_{i\,i}-\sum_{i}e_{i\,i}\otimes e_{i'\,i'}+2\sum_{i<j}e_{j\,i}\otimes e_{j\,i}-2\sum_{i<j}\epsilon_i\epsilon_je_{j\,i}\otimes e_{i'\,j'}.
        $$
It can be verified that the Poisson brackets (\ref{poissonre}) can be uniformly expressed in matrix form. 
To accommodate the $R$-matrix of type C, 
we adopt the approach outlined by A. Molev and E. Ragoucy in\cite{M08}, with some modifications.
    \end{remark}
    \begin{proposition}
    It holds in the Poisson algebra $\mathcal{P}_{n}$,
    \begin{equation} \label{pcen}
    \sum_{k=j,k\neq i'}^{i-1}\epsilon_{i}\epsilon_{k}a_{k'\,i'}a_{k\,j}+\epsilon_{i'}a_{i\,j}+\delta_{j\leq i'\leq i}\epsilon_{i'} a_{i'\,j}=0 ,\quad i\neq j.
    \end{equation}
    Moreover, $a_{k\,l}, k>l>k'$ can be spanned by $a_{i\,j}$, $j'>i>j$ .
\end{proposition}
    \begin{proof}
    It can be immediately deduced from (\ref{scenspecific2}) and direct calculations. 
    \end{proof}

We refer to the action of the braid group $\mathcal{B}_n$ on coideal subalgebras of quantized enveloping algebras  as presented in [\citenum{KP11}, Theorem 3.3] without providing the detailed proofs. 
We will verify the action of the braid group $\mathcal{B}_n$ on $U_q^{tw}(\mathfrak{gl}_n)$ in detail. 
For convenience in the calculation, in the expression of $\beta(s_{i\,j})$, $j'\geq i>j$, the elements $s_{k\,l}$ belonging to $\Omega_2$ are used. 
Actually, as shown in the Lemma \ref{sijspan}, $s_{k\,l}$ can be spanned by $s_{i\,j}$. 
In addition, we explicitly describe the action of $\mathcal{B}_n$ on the elements of the Poisson algebra $\mathcal{P}_n$, which preserves the Poisson bracket.
    \begin{lemma}
    	For $k, m=1,\cdots,n$, the assignment 
    	\begin{align*}
    	 \beta_k : s_{k+1}&\mapsto \frac{1}{q_k-q_k^{-1}}(q_ks_{k+1}s_k-s_ks_{k+1}), & k\neq n-1, n,\\
    		s_{k-1}&\mapsto \frac{1}{q_k-q_k^{-1}}(s_{k}s_{k-1}-q_ks_{k-1}s_{k}), & k\neq 1, n, \\
    		s_k&\mapsto -s_k,&\\
    		s_m&\mapsto s_m, & m\neq k-1, k, k+1,\\
    	\beta_{n-1} : s_{n}&\mapsto \frac{-[2]^{-1}}{(q-q^{-1})^2}(s_{n-1}^2s_{n}-q[2]s_{n-1}s_{n}s_{n-1}+q^2s_{n}s_{n-1}^2)-s_{n},&\\
            \beta_{n} : s_{n-1}&\mapsto \frac{\sqrt{-1}}{q^2-q^{-2}}(s_{n}s_{n-1}-q^2s_{n-1}s_{n}),&
    	\end{align*}
    	defines a series of automorphisms of the algebra $U_q^{tw}(\mathfrak{gl}_n)$.
    \end{lemma}
    \begin{proof}
We will first verify that the images of the generators $s_1,\cdots,s_n$ under $\beta_k$ satisfy the defining relations of  $U_q^{tw}(\mathfrak{gl}_n)$.
Specifically, if $k\neq n-1$, the situation is analogous to the case of  $U_q^{tw}(\mathfrak{o}_n)$, which was verified in [\citenum{M08}, Theorem 2.1]. 
The only nontrivial calculation required is to verify that the images of the pairs of generators $\beta_{n-1}(s_{k})$ and $\beta_{n-1}(s_{k+1})$ with $k=n-3, n-2, n-1, n$ satisfy both Serre relations and the images $\beta_{n-1}(s_{k})$ and $\beta_{n-1}(s_{k+2})$ commute. 
The algebra $U_q^{tw}(\mathfrak{gl}_n)$ is generated by the elements $s_1,s_2,\cdots, s_n$ subject only to the Serre relations 
    \begin{align*}
       &s_ks^2_{k+1}-(q+q^{-1})s_{k+1}s_ks_{k+1}+s^2_{k+1}s_k=-q^{-1}(q-q^{-1})^2s_{k},\qquad k\neq n-1, &\\
       &s^2_ks_{k+1}-(q+q^{-1})s_{k}s_{k+1}s_{k}+s_{k+1}s^2_k=-q^{-1}(q-q^{-1})^2s_{k+1},\qquad k\neq n-1,& \\
        &s_{n-1}s^2_{n}-(q^2-q^{-2})s_{n}s_{n-1}s_{n}+s^2_{n}s_{n-1}=q^{-2}(q^2-q^{-2})^2s_{n-1}, &\\
        &s^3_{n-1}s_n-(1+q^2+q^{-2})s^2_{n-1}s_ns^2_{n-1}-s_ns^3_{n-1}=-q^{-1}(q^2-q^{-2})^2(s_{n-1}s_{n}-s_ns_{n-1}). &
    \end{align*}
It follows from equation (\ref{sre}) that
    \begin{align}
        \beta_{n-1}: 
        s_{n-3}&\mapsto s_{n+4\, n+3},\\
        s_{n-2}&\mapsto s_{n+3\, n+1},\\
        s_{n-1}&\mapsto -s_{n+2\,n+1},\\
        s_{n}&\mapsto -qs_{n+2\,n+1}s_{n+1\,n-1}-qs_{n+2\,n-1}.\label{betasn}
    \end{align}
Hence, for $\beta_{n-1}(s_{n-3})$ and $\beta_{n-1}(s_{n-2})$,   it is straightforward to verify from equation (\ref{sre}) that 
    \begin{align*}
    &s_{n+4\,n+3}s^2_{n+3\,n+1}-(q+q^{-1})s_{n+3\,n+1}s_{n+4\,n+3}s_{n+3\,n+1}+s^2_{n+3\,n+1}s_{n+4\,n+3}\\
    &\qquad =-q^{-1}(q-q^{-1})^2s_{n+3\,n+1}.
    \end{align*}
Similarly, we obtain
    \begin{align*}
    &s^2_{n+4\,n+3}s_{n+3\,n+1}-(q+q^{-1})s_{n+4\,n+3}s_{n+3\,n+1}s_{n+4\,n+3}+s_{n+3\,n+1}s^2_{n+4\,n+3}\\
    &\qquad =-q^{-1}(q-q^{-1})^2s_{n+4\,n+3}.
    \end{align*}
Indeed, a more general relation exists,
    \begin{align*}
    &s_{i\,j}s^2_{j\,l}-(q+q^{-1})s_{j\,l}s_{i\,j}s_{j\,l}+s^2_{j\,l}s_{i\,j}=-q^{-1}(q-q^{-1})^2s_{j\,l},\\
    &s^2_{i\,j}s_{j\,l}-(q+q^{-1})s_{i\,j}s_{j\,l}s_{i\,j}+s_{j\,l}s^2_{i\,j}=-q^{-1}(q-q^{-1})^2s_{i\,j},
\end{align*}
    where $i>j>l>n$. 
  $\beta_{n-1}(s_{n-3})$ commute with both $\beta_{n-1}(s_{n-1})$ and $\beta_{n-1}(s_{n})$ by simple calculation.
We need to verify the relations for the images $\beta_{n-1}(s_{n-2})$ and $\beta_{n-1}(s_{n-1})$, 
    \begin{align} \label{p0}
    &s_{n+3\,n+1}s^2_{n+2\,n+1}-(q+q^{-1})s_{n+2\,n+1}s_{n+3\,n+1}s_{n+2\,n+1}+s^2_{n+2\,n+1}s_{n+3\,n+1}\\ \nonumber
    &\qquad=-q^{-1}(q-q^{-1})^2s_{n+3\,n+1}.
    \end{align}
The left-hand side of \eqref{p0} equals  
    \begin{align} \label{p1}
    &(s_{n+3\,n+1}s_{n+2\,n+1}-qs_{n+2\,n+1}s_{n+3\,n+1})s_{n+2\,n+1}\\ \nonumber
    &\qquad-q^{-1}s_{n+2\,n+1} (s_{n+3\,n+1}s_{n+2\,n+1}-qs_{n+2\,n+1}s_{n+3\,n+1}).
    \end{align}
Indeed, by (\ref{sre}), we have 
    \begin{align*}
    s_{n+3\,n+1}s_{n+2\,n+1}-qs_{n+2\,n+1}s_{n+3\,n+1}=-(q-q^{-1})s_{n+3\,n+2},
    \end{align*}
    which implies that the equation (\ref{p1}) can be written as
    \begin{align*}
    -(q-q^{-1})(s_{n+3\,n+2}s_{n+2\,n+1}-q^{-1}s_{n+2\,n+1}s_{n+3\,n+2})=-q^{-1}(q-q^{-1})^2s_{n+3\,n+1}.
    \end{align*}
Similarly, the following equation holds
     \begin{align*}
    &s^2_{n+3\,n+1}s_{n+2\,n+1}-(q+q^{-1})s_{n+3\,n+1}s_{n+2\,n+1}s_{n+3\,n+1}+s_{n+2\,n+1}s^2_{n+3\,n+1}\\
    &\qquad=-q^{-1}(q-q^{-1})^2s_{n+2\,n+1}.
    \end{align*}
Next we  prove that $\beta_{n-1}(s_{n-2})$ and $\beta_{n-1}(s_{n})$ commute.
    Specifically, we need to verify that
    \begin{align} \label{p2}
    s_{n+3\,n+1}(qs_{n+2\,n+1}s_{n+1\,n-1}+qs_{n+2\,n-1})=(qs_{n+2\,n+1}s_{n+1\,n-1}+qs_{n+2\,n-1})s_{n+3\,n+1}.
    \end{align}
By equation (\ref{sre}), one has
    \begin{align*}
    s_{n+3\,n+1}s_{n+2\,n+1}-qs_{n+2\,n+1}s_{n+3\,n+1}=(q-q^{-1})s_{n+3\,n+2},
    \end{align*}
    and
    \begin{align*}
    s_{n+3\,n+1}s_{n+2\,n-1}-qs_{n+2\,n-1}s_{n+3\,n+1}=(q-q^{-1})(s_{n+2\,n+1}s_{n+3\,n-1}-s_{n+3\,n+2}s_{n+1\,n-1}).
    \end{align*}
Thus, (\ref{p2}) becomes
    \begin{align*}
    &q^2s_{n+2\,n+1}s_{n+3\,n+1}s_{n+1\,n-1}+q(q-q^{-1})s_{n+3\,n+2}s_{n+1\,n-1}-qs_{n+2\,n+1}s_{n+1\,n-1}s_{n+3\,n+1}\\
    &\qquad+q(q-q^{-1})(s_{n+2\,n+1}s_{n+3\,n-1}-s_{n+3\,n+2}s_{n+1\,n-1})\\
    &=qs_{n+2\,n+1}(qs_{n+3\,n+1}s_{n+1\,n-1}-s_{n+1\,n-1}s_{n+3\,n+1})+q(q-q^{-1})s_{n+2\,n+1}s_{n+3\,n-1}=0.
    \end{align*}
Moreover, the Serre relations for the images $\beta_{n-1}(s_{n-1})$ and $\beta_{n-1}(s_n)$ are more complex than those for the other generators.
A detailed verification is required to show that
     \begin{align}
   &\beta^2_{n-1}(s_{n})\beta_{n-1}(s_{n-1})-(q^2+q^{-2})\beta_{n-1}(s_{n})\beta_{n-1}(s_{n-1})\beta_{n-1}(s_{n})+\beta_{n-1}(s_{n-1})\beta^2_{n-1}(s_{n})\nonumber\\ 
   &\quad =-q^{-2}(q^2-q^{-2})^2\beta_{n-1}(s_{n-1}), \label{betan-1sn} \\
   &\beta^3_{n-1}(s_{n-1})\beta_{n-1}(s_n)-(1+q^2+q^{-2})\beta^2_{n-1}(s_{n-1})\beta_{n-1}(s_n)\beta_{n-1}(s_{n-1})\nonumber\\
   &\quad+(1+q^2+q^{-2})\beta_{n-1}(s_{n-1})\beta_{n-1}(s_n)\beta^2_{n-1}(s_{n-1})-\beta_{n-1}(s_n)\beta^3_{n-1}(s_{n-1})\nonumber\\
    &=-q^{-1}(q^2-q^{-2})^2(\beta_{n-1}(s_{n-1})\beta_{n-1}(s_n)-\beta_{n-1}(s_n)\beta_{n-1}(s_{n-1})).\label{betan-1n-1}
    \end{align}
Given the expression
    \begin{align}
    &\beta_{n-1}(s_{n-1})\beta_{n-1}(s_n)-q^{-2}\beta_{n-1}(s_n)\beta_{n-1}(s_{n-1})\nonumber\\
     &\quad=\frac{[2]^{-1}}{(q-q^{-1})^2}(s^3_{n-1}s_n-(1+q^2+q^{-2})s^2_{n-1}s_ns_{n-1}+(1+q^2+q^{-2})s_{n-1}s_ns^2_{n-1}-s_ns^3_{n-1})\nonumber\\
     &\qquad +s_{n-1}s_n-q^{-2}s_ns_{n-1},\label{beta_n+2,n}
    \end{align}   
and using equation (\ref{sn-1sn}), the right-hand side of \eqref{beta_n+2,n} simplifies to
    \begin{align*}
    -q^{-2}s_{n-1}s_n+s_ns_{n-1}=q^{-2}(q^2-q^{-2})s_{n+1\,n-1}.
    \end{align*} 
Then, (\ref{betan-1sn}) becomes 
    \begin{align*}
    &\beta_{n-1}(s_{n})(\beta_{n-1}(s_{n})\beta_{n-1}(s_{n-1})-q^{2}\beta_{n-1}(s_{n-1})\beta_{n-1}(s_{n}))\\
    &\qquad -q^{-2}(\beta_{n-1}(s_{n})\beta_{n-1}(s_{n-1})-q^{2}\beta_{n-1}(s_{n-1})\beta_{n-1}(s_{n}))\beta_{n-1}(s_{n})\\
    &=-(q^2-q^{-2})(\beta_{n-1}(s_n)s_{n+1\,n-1}-q^{-2}s_{n+1\,n-1}\beta_{n-1}(s_n)).\\
    \end{align*}
Using (\ref{sre}) and (\ref{betasn}), the left-hand side of (\ref{betan-1sn}) can be expressed as 
    \begin{align*}
    &(q^2-q^{-2})(qs_{n+2\,n+1}s^2_{n+1\,n-1}+qs_{n+2\,n-1}s_{n+1,n-1}\\
    &\qquad -q^{-1}s_{n+1\,n-1}s_{n+2\,n+1}s_{n+1\,n-1}-q^{-1}s_{n+1\,n-1}s_{n+2\,n-1})\\
    &=(q^2-q^{-2})(qs_{n+2\,n+1}s_{n+1\,n-1}-q^{-1}s_{n+1\,n-1}s_{n+2\,n+1})s_{n+1\,n-1})\\
    &\qquad+(q^2-q^{-2})(qs_{n+2\,n-1}s_{n+1\,n-1}-q^{-1}s_{n+1\,n-1}s_{n+2\,n-1})=(q^2-q^{-2})^2q^{-2}s_{n+2\,n+1}.
    \end{align*}
So (\ref{betan-1sn}) is verified.
    Then for (\ref{betan-1n-1}), the left-hand side is equal to
     \begin{align*}
    &\frac{[2]^{-1}}{(q-q^{-1})^2}(s^3_{n-1}(s^2_{n-1}s_n-(q^2+1)s_{n-1}s_{n}s_{n-1})+q^2s_ns^2_{n-1})\\
    &\quad-(1+q^2+q^{-2})s^2_{n-1}(s^3_{n-1}(s^2_{n-1}s_n-(q^2+1)s_{n-1}s_{n}s_{n-1})+q^2s_ns^2_{n-1})s_{n-1})\\
    &\quad+(1+q^2+q^{-2})s_{n-1}(s^2_{n-1}s_n-(q^2+1)s_{n-1}s_{n}s_{n-1})+q^2s_ns^2_{n-1})s^2_{n-1}\\
    &\quad-(s^2_{n-1}s_n-(q^2+1)s_{n-1}s_{n}s_{n-1})+q^2s_ns^2_{n-1})s^3_{n-1})\\
    &\quad +s^3_{n-1}s_n-(1+q^2+q^{-2})s^2_{n-1}s_ns_{n-1}+(1+q^2+q^{-2})s_{n-1}s_ns^2_{n-1}-s_ns^3_{n-1}\\
     &=q^{-1}(q+q^{-1})(s^2_{n-1}(s_ns_{n-1}-s_{n-1}s_n)-(q^2+1)s_{n-1}(s_ns_{n-1}-s_{n-1}s_n)s_{n-1}\\
     &\quad +q^2(s_ns_{n-1}-s_{n-1}s_n)s^2_{n-1} +s^3_{n-1}s_n-(1+q^2+q^{-2})s^2_{n-1}s_ns_{n-1}\\
     &\quad +(1+q^2+q^{-2})s_{n-1}s_ns^2_{n-1}-s_ns^3_{n-1}\\
    &=-q^{-1}(q^2-q^{-2})^2(\beta_{n-1}(s_{n-1})\beta_{n-1}(s_n)-\beta_{n-1}(s_n)\beta_{n-1}(s_{n-1})).
    \end{align*}
   Thus, equation (\ref{betan-1n-1}) is confirmed.
Each $\beta_k$, for $k=1,\cdots, n$, defines a homomorphism $U_q^{tw}(\mathfrak{gl}_n)$.
Furthermore, each $\beta_k$ is invertible and is given by  
   \begin{align*}
    		\beta_k^{-1} : s_{k+1}&\mapsto \frac{1}{q_k-q_k^{-1}}(s_{k+1}s_k-q_ks_ks_{k+1}), & k\neq n-1,\\
    		s_{k-1}&\mapsto \frac{1}{q_k-q_k^{-1}}(q_ks_{k}s_{k-1}-s_{k-1}s_{k}), & k\neq n,\\
    		s_k&\mapsto -s_k,&\\
    		s_m&\mapsto s_m, & m\neq k-1, k, k+1,\\
    		\beta_{n-1}^{-1} : s_{n}&\mapsto \frac{-[2]^{-1}}{(q-q^{-1})^2}(s_{n}s_{n-1}^2-q[2]s_{n-1}s_{n}s_{n-1}+q^2s_{n-1}^2s_{n})-s_{n},&\\
            \beta_{n}^{-1} : s_{n-1}&\mapsto \frac{\sqrt{-1}}{q^2-q^{-2}}(q^2s_{n}s_{n-1}-s_{n-1}s_{n}).
    	\end{align*}
    \end{proof}
    \begin{lemma}
The actions of the automorphism $\beta_{n-1}$ and $\beta_n$ are given by 
      \begin{align}
        \beta_{n-1}:&s_{n+1\,n-1}\mapsto  -s_{n+2\,n},\quad  s_{n+2\,n}\mapsto -s_{n+1\,n-1}, \\
        &s_{n+2\,n-1}\mapsto q^2s_{n\,n-1}s_{n+1\,n-1}+(q^2-1)s_{n+2\,n-1}-qs_{n+1\,n},\label{betan-1sn+2n-1}\\
        \beta_{n}:&s_{n\,n-1}\mapsto\sqrt{-1}s_{n+2\,n},\quad s_{n+1\,n-1}\mapsto-\sqrt{-1}s_{n\,n-1},\\
        &s_{n+2\,n-1}\mapsto-s_{n+2\,n-1} \label{betansn+2n_1}.
    \end{align}
    \end{lemma}
    \begin{proof}
In fact, it follows from equation (\ref{sre}) that
      \begin{align*}
    \beta_{n-1}(s_{n+2\,n})=\frac{1}{(q^2-q^{-2})}(\beta_{n-1}(s_{n})\beta_{n-1}(s_{n-1})-q^2\beta_{n-1}(s_{n-1})\beta_{n-1}(s_{n})).
    \end{align*}
Due to (\ref{beta_n+2,n}), we get  $\beta_{n-1}(s_{n+2\,n})=-s_{n+1\,n-1}.$
A similar calculation gives
     \begin{align*}
    s_{n+1\,n-1}=\frac{1}{(q^2-q^{-2})}(-q^2\beta_{n-1}^{-1}(s_{n})\beta_{n-1}^{-1}(s_{n-1})+\beta_{n-1}^{-1}(s_{n-1})q^2\beta_{n-1}^{-1}(s_{n})).
    \end{align*}
Subsequently, it follows that $\beta_{n-1}(s_{n+1\,n-1})=-s_{n+2\,n}$. 
To verify (\ref{betan-1sn+2n-1}), using (\ref{sre}), we find
     \begin{align*}
    &(q-q^{-1})(q^2-q^{-2})q^{-1}(\beta_{n-1}(s_{n+2\,n-1})-q^{-1}\beta_{n-1}(s_{n})) \nonumber \\
    &\qquad=\beta_{n-1}(s^2_{n-1}s_n-(q^2-q^{-2})s_{n-1}s_ns_{n-1}+s_ns^2_{n-1}).
     \end{align*}
   Further, it can be obtained directly that
    \begin{align*}
    &\beta_{n-1}(s^2_{n-1}s_n-(q^2-q^{-2})s_{n-1}s_ns_{n-1}+s_ns^2_{n-1}) \nonumber \\
    &\qquad =-q^{-1}(q^2-q^{-2})^2(\beta_{n-1}(s_n)+s_n)-(q-q^{-1})(q^2-q^{-2})q^{-1}(s_{n+2\,n-1}-q^{-1}s_n),
     \end{align*}
so one gets
    \begin{align*}
\beta_{n-1}(s_{n+2\,n-1})=&-q(\beta_{n-1}(s_n)+s_{n})-s_{n+2\,n-1}\\
=& q^2s_{n\,n-1}s_{n+1\,n-1}+(q^2-1)s_{n+2\,n-1}-qs_{n+1\,n}.
    \end{align*}
     Moreover, we can easily get $\beta_n(s_{n-1})=\sqrt{-1}s_{n+2\,n}$.
Since 
    \begin{align*}
    \beta_n^{-1}(s_{n\,n-1})=\frac{\sqrt{-1}}{(q^2-q^{-2})}(q^2s_{n+1\,n}s_{n\,n-1}-s_{n\,n-1}s_{n+1\,n})=\sqrt{-1}s_{n+1\,n-1},
    \end{align*}
    it follows that $\beta_{n}(s_{n+1\,n-1})=-\sqrt{-1}s_{n\,n-1}$.
Then, for (\ref{betansn+2n_1}), by (\ref{sre}), one has
     \begin{align*}
    (q-q^{-1})\beta_n(s_{n+2\,n-1}-q^{-1}s_{n+1\,n})=\beta_n(q^{-1}s_{n+1\,n-1}s_{n\,n-1}-qs_{n\,n-1}s_{n+1\,n-1}),
     \end{align*}
     which implies
     \begin{align*}
    (q-q^{-1})(\beta_n(s_{n+2\,n-1})+q^{-1}s_{n+1\,n})&=q^{-1}s_{n\,n-1}s_{n+2\,n}-qs_{n+2\,n}s_{n\,n-1}\\
    &=(q-q^{-1})q^{-1}s_{n+1\,n}-(q-q^{-1})s_{n+2\,n-1}.
     \end{align*}
     Thus, we get $\beta_n(s_{n+2\,n-1})=-s_{n+2\,n-1}$. The lemma is verified.
    \end{proof}
      \begin{theorem}
    For $k=1,\cdots,n$, $\beta_k$ defines an action of the braid group $\mathcal{B}_n$ on $U_q^{tw}(\mathfrak{gl}_n)$.
    \end{theorem}
    \begin{proof}
We have verified that $\beta_k$ defines an automorphism on $U_q^{tw}(\mathfrak{gl}_n)$.
In addition, we need to verify that the automorphisms $\beta_k$ satisfy the braid group relations. 
For each generator $s_k$, it is necessary to confirm that 
    \begin{align}
    &\beta_k\beta_{k+1}\beta_k(s_i)=\beta_{k+1}\beta_k\beta_{k+1}(s_i),& k\neq n-1, n,\label{betakbetak+1}\\
        &\beta_n\beta_{n-1}\beta_n\beta_{n-1}(s_{i})=\beta_{n-1}\beta_n\beta_{n-1}\beta_n(s_{i}),\label{betanbetan-1}\\
        &\beta_k\beta_l(s_i)=\beta_{l}\beta_{k}(s_i), &|k-l|>1,\label{betakbetal}
    \end{align}
    for $k=1,\cdots, n-1$. 
Actually, the proofs of (\ref{betakbetak+1}) and (\ref{betakbetal}) are similar to the cases of type A in [\citenum{M08}, Theorem 2.1].
The only nontrivial case is (\ref{betanbetan-1}).
    Thus, for $k=n-1$, we have   
     \begin{align*}
     &\beta_n\beta_{n-1}\beta_n\beta_{n-1}(s_{n-1})=\beta_n\beta_{n-1}\beta_n(-s_{n-1})=\beta_n\beta_{n-1}(-\sqrt{-1}s_{n+2\,n})=s_{n-1},\\
     &\beta_{n-1}\beta_n\beta_{n-1} \beta_n(s_{n-1})=\beta_{n-1}\beta_n\beta_{n-1}(\sqrt{-1}s_{n+2\,n})=\beta_{n-1}\beta_n(-\sqrt{-1}s_{n+1\,n-1})=s_{n-1}.
    \end{align*}
    Moreover, for $k=n$, it follows that
       \begin{align*}
    \beta_{n-1}\beta_n\beta_{n-1}\beta_n(s_{n})&=\beta_{n-1}\beta_n\beta_{n-1}(-s_{n})=\beta_{n-1}\beta_n(qs_{n\,n-1}s_{n+1\,n-1}+qs_{n+2\,n-1}\\
    &=\beta_{n-1}(qs_{n+2\,n}s_{n\,n-1}-qs_{n+2\,n-1}))\\
    &=qs_{n+1\,n-1}s_{n\,n-1}-q^3s_{n\,n-1}s_{n+1\,n-1}-q(q^2-1)s_{n+2\,n-1}+q^2s_{n+1\,n},\\
    \beta_n\beta_{n-1}\beta_n\beta_{n-1}(s_{n})&=-\beta_n(qs_{n+1\,n-1}s_{n\,n-1}-q^3s_{n\,n-1}-q(q^2-1)s_{n+2\,n-1}+q^2s_{n+1\,n})\\
    &=-qs_{n\,n-1}s_{n+2\,n}+q^3s_{n+2\,n}s_{n\,n-1}-q(q^2-1)s_{n+2\,n-1}+q^2s_{n+1\,n}.
    \end{align*}
    Using (\ref{sre}), we observe that
    \begin{align*}
        q^{-1}s_{n+1\,n-1}s_{n\,n-1}-qs_{n\,n-1}s_{n+1\,n-1}=-q^{-1}s_{n\,n-1}s_{n+2\,n}+qs_{n+2\,n}s_{n\,n-1},
    \end{align*}
    which implies that  $\beta_{n-1}\beta_n\beta_{n-1}\beta_n(s_{n})= \beta_n\beta_{n-1}\beta_n\beta_{n-1}(s_{n})$.
    \end{proof}
To define the action of the braid group $\mathcal{B}_n$ on the Poisson algebra $\mathcal{P}_n$, we first examine the action of the braid group on $s_{i\,j}$, and then degenerate this action to $a_{i\,j}$.
    \begin{theorem}\label{betas}
In terms of the generators $s_{i\,j}$ of the algebra $U_q^{tw}(\mathfrak{gl}_n)$ with  $i> j\geq j'$, for each index $k=1,\cdots,n-2$, the action of $\beta_{k}$ and  $\beta_{n-1}$, $\beta_{n}$ is given by
    \begin{align}
    \beta_{k}:
    s_{k+1\,k}&\mapsto -s_{k+1\,k},\label{ksk+1k}&&\\
    s_{l\,k}&\mapsto q^{-1}s_{l\,k+1}-s_{l\,k}s_{k+1\,k},\quad s_{l\,k+1}\mapsto s_{l\,k}, &k+2\leq l\leq k'-2,& \label{slk}\\ 
    s_{k\,l}&\mapsto qs_{k+1\,l}-qs_{k+1\,k}s_{k\,l},\quad \  s_{k+1\,l}\mapsto s_{k\,l},\label{ksk+1l} & l\leq k-1,&\\
    s_{k'-1\,k+1}&\mapsto q^{-1}s_{k'-1\,k}s_{k+1\,k}+q^{-1}s_{k'\,k}+(1-q^{-2})s_{k'-1\,k+1},&&\label{sk'-1k+1}\\
    s_{k'-1\,k}&\mapsto q^{\overline{k'-1}-\overline{k}+1}s_{k'\,k}-q^{2\overline{k+1}+1}\beta_k(s_{k'-1\,k+1})s_{k+1\,k}\nonumber\\
    &\quad +\sum_{t=k+2}^{k'-2}q^{\overline{k'-1}-\overline{t}+1}\epsilon_ts_{t'\,k}\beta_k(s_{t\,k}),
    \end{align}
    \begin{align}
    s_{k'-1\,l}&\mapsto qs_{k'\,k'-1}s_{k'-1\,l}+qs_{k'\,l}, & l<k,& \label{sk'-1l}\\
    s_{k'\,l}&\mapsto s_{k'-1\,l},& l<k,& \label{sk'l}\\
    s_{k'\,k}&\mapsto -s_{k'-1\,k}s_{k+1\,k}+q^{-1}s_{k'-1\,k+1}, && \label{sk'k}\\
    s_{k'+1\,k-1}&\mapsto s_{k'+1\,k-1}+q^{-1}s_{k'-1\,k}s_{k'\,k'-1}&& \nonumber\\
    &\quad -q^{-2}s_{k'-1\,k+1}-q^{-1}\beta_k(s_{k'\,k}),\\
    s_{m'\,m}&\mapsto s_{m'\,m}-q^{-1}s_{m'-1\,m+1}+q^{-1}\beta_k(s_{m'-1\,m+1}),&m\neq k-1,k,k+1,&\\
     s_{m\,l}&\mapsto s_{m\,l}, &\text{otherwise},& \nonumber
\end{align}
\begin{align}
    \beta_{n-1}:  
     s_{n\,n-1}&\mapsto -s_{n\,n-1},&& \label{snn-1}\\
     s_{n+1\,n-1}&\mapsto -s_{n+2\,n-1},&&\\
     s_{n-1\,l}&\mapsto qs_{n\,l}- qs_{n\,n-1}s_{n-1\,l},\quad s_{n\,l}\mapsto s_{n-1\,l}, &l\leq n-2,&\\
     s_{n+1\,n}&\mapsto -qs_{n+2\,n+1}s_{n+1\,n-1}-qs_{n+2\,n-1},&&\\
     s_{n+2\,n-1}&\mapsto (q^2-1)s_{n+2\,n-1}-qs_{n+1\,n}&& \nonumber \\
     &\quad +q^2s_{n\,n-1}s_{n+1\,n-1},&&\label{sn+2n-1} \\
     s_{n+1\,l}&\mapsto -qs_{n+2\,l}+qs_{n\,n-1}(\sum_{t=n+3}^{l'}q^{\overline{t}-\overline{n+2}}s_{t\,n}s_{t'\,l})& &\nonumber \\
     &\quad +qs_{n+2\,n-1}s_{n-1\,l}-s_{n+1\,n}s_{n-1\,l},&l\leq n-2,&\label{sn+1l}\\
    s_{n+2\,l}&\mapsto -s_{n+1\,l}, &l \leq n-1,&\\
    s_{n+3\,n-2}&\mapsto -s_{n+3\,n-2}-q^{-1}s_{n+1\,n-1}s_{n\,n-1}&& \nonumber \\
    &\quad +q^{-2}s_{n+1\,n}+q^{-1}\beta_{n-1}(s_{n+2\,n-1}),&&\\
    s_{m'\,m}&\mapsto -s_{m'\,m}+q^{-1}s_{m'-1\,m+1}\nonumber\\
    &\quad +q^{-1}\beta_{n-1}(s_{m'-1\,m+1}), &m\geq n+4&,\\
    s_{m\,l}&\mapsto -s_{m\,l}, &m\geq n+4, l\leq n-3,&\\
     s_{m\,l}&\mapsto s_{m\,l}, &\text{otherwise},&
\end{align}
\begin{align}
     \beta_{n}:
      s_{n+1\,n}&\mapsto -s_{n+1\,n}, &&\\
      s_{n\,l}&\mapsto\sqrt{-1}(qs_{n+1\,l}-qs_{n+1\,n}s_{n\,l}),\ s_{n+1\,l}\mapsto -\sqrt{-1}s_{n\,l}, &l\leq n-1,&\\ 
      s_{n+2\,n-1}&\mapsto -s_{n+2\,n-1},&&\\
      s_{m'\,m}&\mapsto -s_{m'\,m}+q^{-1}s_{m'-1\,m+1}-q^{-1}\beta_n(s_{m'-1\,m+1}),&m\geq n+3,&\\
      s_{m\,l}&\mapsto s_{m\,l}, &\text{otherwise.}&
    \end{align}
    \end{theorem}
    \begin{proof}
From (\ref{ksk+1k})--(\ref{ksk+1l}), the cases are similar to  [\citenum{M08}, Corollary 2.2] of type A.
Since $\beta_k(s_m)=s_m$ for $k\neq k-1,k,k+1$, it follows that $\beta_k(s_{m\,l})=s_{m\,l}$, for $m>l\geq m'$, $m\geq k+3$. 
For $k\leq n-2$, given that
    \begin{align*}
    q^{-1}s_{k'-2\,k+1}s_{k+2\,k+1}-qs_{k+2\,k+1}s_{k'-2\,k+1}=-(q-q^{-1})(q^{-1}s_{k'-2\,k+2}+s_{k'-1\,k+1}),
    \end{align*} 
we find that 
    \begin{align*}
    (q-q^{-1})\beta_k(s_{k'-1\,k+1})=q^{-1}s_{k'-2\,k}s_{k+2\,k}-qs_{k+2\,k}s_{k'-2\,k}+(q-q^{-1})q^{-1}s_{k'-2\,k+2}.
    \end{align*}
Using equation (\ref{sre}), we deduce
    \begin{align*}
    &q^{-1}s_{k'-2\,k}s_{k+2\,k}-qs_{k+2\,k}s_{k'-2\,k}=\\
    &\qquad -(q-q^{-1})q^{-1}(s_{k'-2\,k+2}-s_{k'-1\,k}s_{k+1\,k}-s_{k'\,k}-(q-q^{-1})s_{k'-1\,k+1}),
    \end{align*}
which confirms that (\ref{sk'-1k+1}) holds.
From the relation
    \begin{align*}
    \beta_k^{-1}(s_{k+1})=\frac{1}{(q-q^{-1})}(s_{k+1}s_{k}-qs_{k}s_{k+1}),
    \end{align*}
and using $s_k=s_{k'}$, we get $\beta_k(s_{k'\,k'-2})=s_{k'-1\,k'-2}$. 
The equation implies $\beta_k(s_{k'\,k+1})=s_{k'-1\,k}$. 
By 
    \begin{align*}
    q^{\overline{k}-\overline{k'-1}}s_{k'-1\,k}=qs_{k'\,k+1}+\sum_{t=k+1}^{k'-2}q^{\overline{k}-\overline{t}+1}\epsilon_{t}s_{t'\,k+1}s_{t\,k},
    \end{align*}
it follows that
    \begin{align*}
    \beta_{k}(s_{k'-1\,k})=&q^{\overline{k'-1}-\overline{k}+1}s_{k'\,k}-q^{2\overline{k+1}+1}\beta_k(s_{k'-1\,k+1})s_{k+1\,k}+\sum_{t=k+2}^{k'-2}q^{\overline{k'-1}-\overline{t}+1}\epsilon_ts_{t'\,k}\beta_k(s_{t\,k}),
    \end{align*}
    We can easily obtain $\beta_{k}(s_{k'-2\,l})=s_{k'-2\,l}$, $l<k$, and 
    \begin{align*}
    \beta_{k}(s_{k'-1\,k'-2})=qs_{k'\,k'-2}+qs_{k'\,k'-1}s_{k'-1\,k'-2}.
    \end{align*}
    However, for $l<k$, we have 
    \begin{align*}
    \beta_{k}(s_{k'-1\,l})&=\frac{-1}{(q-q^{-1})}(q \beta_{k}(s_{k'-1\,k'-2})s_{k'-2\,l}-s_{k'-2\,l} \beta_{k}(s_{k'-1\,k'-2}))\\
    &=\frac{-1}{(q-q^{-1})}(q(qs_{k'\,k'-2}s_{k'-2\,l}-s_{k'-2\,l}s_{k'\,k'-2})\\
    &\quad +q(qs_{k'\,k'-1}s_{k'-1\,k'-2}s_{k'-2\,l}-s_{k'-2\,l}s_{k'\,k'-1}s_{k'-1\,k'-2}))\\
    &=qs_{k'\,l}+qs_{k'\,k'-1}s_{k'-1\,l},
    \end{align*}
    then (\ref{sk'-1l}) is verified.
(\ref{sk'l}) can be directly obtained by 
    \begin{align*}
    (q-q^{-1})\beta_k(s_{k'\,l})=\beta_k(qs_{k'\,k+1}s_{k+1\,l}-s_{k+1\,l}s_{k'\,k+1}).
    \end{align*}
Given the relation from (\ref{sre})
    \begin{align*}
    qs_{k'\,k+1}s_{k+1\,k}-q^{-1}s_{k+1\,k}s_{k'\,k+1}=(q-q^{-1})(s_{k'\,k}-q^{-1}s_{k'-1\,k+1}),
    \end{align*}
the following equation holds
    \begin{align*}
    (q-q^{-1})(\beta_{k}(s_{k'\,k})-q^{-1}\beta_k(s_{k'-1\,k+1})=-qs_{k'-1\,k}s_{k+1\,k}+q^{-1}s_{k+1\,k}s_{k'-1\,k}.
    \end{align*}
    While 
    \begin{align*}
    q^{-1}s_{k'-1\,k}s_{k+1\,k}-qs_{k+1\,k}s_{k'-1\,k}=-(q-q^{-1})(q^{-1}s_{k'-1\,k+1}-s_{k'\,k}),
    \end{align*}
 (\ref{sk'k}) can be obtained by calculation. 
Moreover, (\ref{snn-1})--(\ref{sn+2n-1}) have been verified in the previous text. 
We have proved $\beta_{n-1}(s_{n+2\,n})=-s_{n+1\,n-1}$ in Lemma 6.5. 
We can easily obtain $\beta_{n-1}(s_{l\,n})=-s_{l\,n-1}$, $l>n'$. For $l\leq n-2$, given the relation
    \begin{align*}
    s_{n+1\,n}s_{n\,l}-s_{n\,k}s_{n+1\,n}=(q-q^{-1})(q^{-1}s_{n+1\,l}+\sum_{t=n+2}^{l'}q^{\overline{t}-\overline{n+1}}s_{t\,n}s_{t'\,l}),
    \end{align*}
it follows that
    \begin{align*}
    \beta_{n-1}(s_{n+1\,l})=\frac{q}{(q-q^{-1})}\beta_{n-1}(s_{n+1\,n}s_{n\,l}-s_{n\,k}s_{n+1\,n})-\sum_{t=n+2}^{l'}q^{\overline{t}-\overline{n+1}+1}s_{t\,n-1}\beta_{n-1}(s_{t'\,l}).
    \end{align*}
(\ref{sn+1l}) can be confirmed by direct calculation. 
Other remaining cases can be verified similarly using induction. 
    \end{proof}
    \begin{corollary}
For each index $k=1,\cdots,n-2$, the braid group actions of $\beta_{k}$ and  $\beta_{n-1}$, $\beta_{n}$ on the Poisson algebra $\mathcal{P}_n$ are given by
    \begin{align*}
    \beta_{k}:
    a_{k+1\,k}&\mapsto -a_{k+1\,k},&&\\
    a_{l\,k}&\mapsto a_{l\,k+1}-a_{l\,k}a_{k+1\,k},\qquad a_{l\,k+1}\mapsto a_{l\,k}, &k+2\leq l\leq k'-2,& \\ 
    a_{k\,l}&\mapsto a_{k+1\,l}-a_{k+1\,k}a_{k\,l},\qquad \  a_{k+1\,l}\mapsto a_{k\,l}, & l\leq k-1,&\\
    a_{k'-1\,k+1}&\mapsto a_{k'-1\,k}a_{k+1\,k}+a_{k'\,k},&&\\
    a_{k'-1\,k}&\mapsto a_{k'\,k}-\beta_k(a_{k'-1\,k+1})a_{k+1\,k}
    +\sum_{t=k+2}^{k'-2}\epsilon_ta_{t'\,k}\beta_k(a_{t\,k}),\\
    a_{k'-1\,l}&\mapsto a_{k'\,k'-1}a_{k'-1\,l}+a_{k'\,l},& l<k,& \\
    a_{k'\,l}&\mapsto a_{k'-1\,l},& l<k,& \\
    a_{k'\,k}&\mapsto -a_{k'-1\,k}a_{k+1\,k}+a_{k'-1\,k+1}, && \\
    a_{m'\,m}&\mapsto a_{m'\,m}-a_{m'-1\,m+1}+\beta_k(a_{m'-1\,m+1}),&m\neq k-1,k,k+1,&\\
    a_{m\,l}&\mapsto a_{m\,l}, &\text{otherwise},&\\
    \beta_{n-1}:  
     a_{n\,n-1}&\mapsto -a_{n\,n-1},&& \\
     a_{n+1\,n-1}&\mapsto -a_{n+2\,n-1},&&\\
     a_{n-1\,l}&\mapsto a_{n\,l}- a_{n\,n-1}a_{n-1\,l},\quad a_{n\,l}\mapsto a_{n-1\,l}, &l\leq n-2,&\\
     a_{n+1\,n}&\mapsto -a_{n+2\,n+1}a_{n+1\,n-1}-a_{n+2\,n-1},&&\\
     a_{n+2\,n-1}&\mapsto a_{n\,n-1}a_{n+1\,n-1}-a_{n+1\,n},&&  \\
     a_{n+1\,l}&\mapsto -a_{n+2\,l}+\sum_{t=n+3}^{l'}a_{n\,n-1}a_{t\,n}a_{t'\,l}
    +(a_{n+2\,n-1}-a_{n+1\,n})a_{n-1\,l},&l\leq n-2,&\\
    a_{n+2\,l}&\mapsto -a_{n+1\,l}, &l \leq n-1,&\\
    a_{n+3\,n-2}&\mapsto -a_{n+3\,n-2}-a_{n+1\,n-1}a_{n\,n-1}+ a_{n\,n-1}a_{n+1\,n-1},&&\\
     a_{n+3\,l}&\mapsto-a_{n+3\,l}, \quad a_{m\,l}\mapsto -a_{m\,l}, &m\geq n+4, l\leq n-3,&\\
    a_{m'\,m}&\mapsto -a_{m'\,m}+a_{m'-1\,m+1}+\beta_{n-1}(a_{m'-1\,m+1}), &m\geq n+4&,\\
     a_{m\,l}&\mapsto a_{m\,l}, &\text{otherwise},&\\
     \beta_{n}:
      a_{n+1\,n}&\mapsto -a_{n+1\,n}, &&\\
      a_{n\,l}&\mapsto  \sqrt{-1}(a_{n+1\,l}-a_{n+1\,n}a_{n\,l}),\quad a_{n+1\,l}\mapsto -\sqrt{-1}a_{n\,l}, &l\leq n-1,&\\
      a_{n+2\,n-1}&\mapsto -a_{n+2\,n-1},&&\\
      a_{m'\,m}&\mapsto -a_{m'\,m}+a_{m'-1\,m+1}-\beta_n(a_{m'-1\,m+1}),&m\geq n+3,&\\
      a_{m\,l}&\mapsto a_{m\,l}, &\text{otherwise.}&
    \end{align*}
  \end{corollary}
  \begin{proof}
    The corollary follows immediately from Theorem \ref{betas}.
  \end{proof}
  \begin{remark}
     Due to the complexity of Poisson bracket, it is challenging to directly obtain specific relations through the braid group $\mathcal{B}_n$ actions on $a_{i+1\,i}$.
      As an application of the braid group actions on the coideal subalgebra, we derive the specific action on the generators $a_{i\,j}$ of Poisson algebra $\mathcal{P}_n$. The difficulty of the problem arises from the need to consider the limit $q\to 1$, which requires to kill $1/(q-q^{-1})$ using the relations of $s_{i\,j}$. 
  \end{remark}
\section*{Acknowledgments}
H. Zhang is partially supported by the support of the National Natural Science Foundation of China 12271332 and Natural Science Foundation of Shanghai 22ZR1424600.

\end{document}